\documentclass[reqno,11pt,a4paper]{amsart}
\usepackage{leftidx}
\usepackage{xcolor}
\usepackage{amssymb,amsmath}
\usepackage{enumerate}
\usepackage{comment}
\usepackage[all,cmtip]{xy}

\newcommand{\largewedge}{\mbox{\Large $\wedge$}}

\newtheorem{MainTheorem}{Theorem}
\newtheorem{Proposition}{Proposition}[section]
\newtheorem{Definition}[Proposition]{Definition}
\newtheorem{Lemma}[Proposition]{Lemma}
\newtheorem{Theorem}[Proposition]{Theorem}
\newtheorem{Corollary}[Proposition]{Corollary}

\DeclareMathOperator{\Val}{Val}
\DeclareMathOperator{\Curv}{Curv}

\DeclareMathOperator{\vol}{vol}

\DeclareMathOperator{\Kl}{Kl}
\DeclareMathOperator{\Gr}{Gr}

\DeclareMathOperator{\re}{Re}

\DeclareMathOperator{\glob}{glob}

\DeclareMathOperator{\Sp}{Sp}
\DeclareMathOperator{\SO}{SO}
\DeclareMathOperator{\GL}{GL}

 \newcommand{\spsp}{\Sp(2)\Sp(1)}

\newcommand{\spann}{\mathrm{span}}
\newcommand{\R}{\mathbb{R}}
\newcommand{\C}{\mathbb{C}}

\newcommand{\h}{\mathbb{H}}

\newcommand\flag[2]{\left[\begin{array}{c} #1\\ #2 \end{array}
  \right]} 
\title[]{Kinematic formulas on the quaternionic plane} 
\author{Andreas Bernig}
\author{Gil Solanes}
 
\email{bernig@math.uni-frankfurt.de}
\email{solanes@mat.uab.cat}

\address{Institut f\"ur Mathematik, Goethe-Universit\"at Frankfurt,
Robert-Mayer-Str. 10, 60629 Frankfurt, Germany}
\address{Departament de Matem\`atiques, Universitat Aut\`onoma de Barcelona, 08193 Bellaterra, Spain}

\thanks{A.B. was supported by DFG grant BE 2484/5-2. G.S. is a Serra H\'unter Fellow and was supported by FEDER-MINECO grants MTM2012-34834 and UNAB13-4E-1604.\\ AMS 2010 {\it Mathematics subject
classification}: 53C65, 
52A22
}

\begin{document}

\begin{abstract}
We introduce different bases for the vector space of $\spsp$-invariant, translation invariant continuous valuations on the quaternionic plane and determine a complete set of kinematic formulas. 
\end{abstract}

\maketitle
\tableofcontents

\section{Introduction}

\subsection{General background}

One of the most influential theorems in integral geometry is Hadwiger's theorem from 1957 which gives a characterization of the space of continuous translation-invariant and rotation-invariant {\em valuations} on a euclidean vector space $V$. A valuation is a map on the space of all compact convex bodies in $V$ such that $\mu(K \cup L)+\mu(K \cap L)=\mu(K)+\mu(L)$ whenever $K,L,K \cup L$ are compact and convex. Many formulas in Crofton-style integral geometry (kinematic formulas, additive kinematic formulas, Crofton- and Kubota formulas, etc.) are more or less direct consequences of Hadwiger's theorem. 

Thanks to the groundbreaking work of Alesker, it is now known that a Hadwiger-type theorem holds under the weaker assumptions of continuity, translation invariance and invariance under a compact group acting transitively on the unit sphere. More precisely,  let $G$ be a closed subgroup of the orthogonal group $\mathrm{O}(V)$ and let $\Val^G$ be the vector space of all continuous, translation and $G$-invariant valuations. Alesker has shown that $\Val^G$ is finite-dimensional if and only if $G$ acts transitively on the unit sphere. 

The connected groups acting effectively and transitively on some unit sphere were classified by Montgomery-Samelson and Borel \cite{borel49, montgomery_samelson43}. Besides the euclidean rotation group, there are complex and quaternionic versions of rotation groups, $\mathrm{U}(n), \mathrm{SU}(n), \mathrm{Sp}(n), \mathrm{Sp}(n)\mathrm{U}(1), \mathrm{Sp}(n)\mathrm{Sp}(1)$, as well as three exceptional cases: $\mathrm{G}_2, \mathrm{Spin}(7), \mathrm{Spin}(9)$. 

Alesker's theorem gives the finite-dimensionality of $\Val^G$ in each of these cases, but the explicit computation of the dimensions required more efforts. This computation was worked out in \cite{alesker_mcullenconj01, bernig_sun09, bernig_g2, bernig_qig, bernig_voide}. The next step is to find a geometrically meaningful basis of $\Val^G$. For the unitary case $G=\mathrm{U}(n)$, several bases were found in \cite{alesker03_un} and \cite{bernig_fu_hig}. The special unitary group was treated in \cite{bernig_sun09}. For $G=\mathrm{G}_2$ and $G=\mathrm{Spin}(7)$, see \cite{bernig_g2}. In the quaternionic cases $G=\mathrm{Sp}(n), G=\mathrm{Sp}(n)\mathrm{U}(1), G=\mathrm{Sp}(n) \mathrm{Sp}(1)$ as well as in the octonionic case  $G=\mathrm{Spin}(9)$, a construction of a basis of $\Val^G$ is still missing. The only exception is the case $G=\spsp$ acting on the quaternionic plane. In \cite{bernig_solanes}, we used quaternionic linear algebra (in particular the Moore 
determinant of a quaternionic hermitian matrix) to describe a basis of $\Val^{\spsp}$, see below for details on this basis.

As with Hadwiger's classical theorem, the finite dimensionality of $\Val^G$ implies the existence of {\em kinematic formulas}. Given  a basis
 $\phi_1,\ldots,\phi_N$  of $\Val^G$,  there are constants $c_{kl}^i$ such that 
\begin{displaymath}
 \int_{\bar G} \phi_i(K \cap \bar g L) d\bar g=\sum_{k,l} c_{k,l}^i \phi_k(K)\phi_l(L), 
\end{displaymath}
where $K,L$ are compact and convex in $V$, and $\bar G$ is the group generated by $G$ and translations, endowed with a convenient invariant measure. The constants are of course independent of $K,L$. In principle, they could be determined by plugging in easy examples and obtaining a system of equations. However, this method (called template method) becomes too cumbersome even in the unitary case $\mathrm{U}(n)$. In \cite{bernig_fu_hig, fu06}, a more algebraic approach was used. It turns out that the $c_{k,l}^i$ are the structure constants of some cocommutative and coassociative coalgebra.  More precisely, we may define a coproduct by putting
\begin{align*}
 k_G:\Val^G & \to \Val^G \otimes \Val^G, \\
 \phi_i & \mapsto \sum_{k,l} c_{k,l}^i \phi_k \otimes \phi_l.
\end{align*}

Let $\chi \in \Val^G$ be the Euler characteristic (defined by $\chi(K)=1$ for each non-empty $K$). Since $\Val^G$ is canonically self-dual through 
\begin{displaymath}
k_G(\chi)\in\Val^G\otimes\Val^G=\mathrm{Hom}(\Val^{G*},\Val^G), 
\end{displaymath}
we even obtain an algebra structure on $\Val^G$. By \cite{bernig_fu06}, the algebra product is the previously constructed Alesker product of valuations \cite{alesker04_product}. There exist other types of kinematic formulas, called additive kinematic formulas, which 
can be related to the usual kinematic formulas via the Alesker-Fourier transform, 
see \cite{alesker_fourier, bernig_fu06}.

The algebraic approach for the computation of kinematic formulas was also successful for the groups $\mathrm{SU}(n)$ \cite{bernig_sun09}, $\mathrm{G}_2, \mathrm{Spin}(7)$ \cite{bernig_g2}. For the other groups from the list above, the kinematic formulas remain unknown, although some partial results are available (cf. e.g. \cite{gual_naveira_tarrio}). The main result of the present paper is the determination of these formulas in the case $G=\spsp$.

Another recent line of research is to replace valuations by some localized versions, called smooth curvature measures (or local valuations by some authors). A smooth curvature measure is a valuation with values in the space of signed Borel measures on $V$, satisfying some technical conditions given below. If $\Phi_1,\ldots,\Phi_m$ is a basis for the space of translation-invariant, $G$-invariant and smooth curvature measures, then there are constants $\tilde c_{k,l}^i$ such that the following {\em local kinematic formulas}  hold
\begin{displaymath}
 \int_{\bar G} \Phi_i(K \cap \bar g L, U_1 \cap \bar g U_2) d\bar g=\sum_{k,l} \tilde c_{k,l}^i \Phi_k(K,U_1)\Phi_l(L,U_2), 
\end{displaymath}
where $U_1,U_2$ are Borel subsets of $V$. 

In the case $G=\mathrm{SO}(n)$, the coefficients of the local kinematic formulas are  the same as for the global kinematic formulas. In contrast to this, the determination of the local kinematic formulas for $\mathrm{U}(n)$ was only recently obtained \cite{bernig_fu_solanes}. A major point in this study was the transfer from the flat hermitian case $\mathbb{C}^n$ to the curved hermitian case, i.e. the study of (global and local) kinematic formulas on complex projective space  $\mathbb{C}P^n$ and complex hyperbolic space $\mathbb{C}H^n$. The very useful transfer principle states that the local kinematic formulas are the same in these three cases. Somehow unexpectedly, it turned out that there is also a kind of transfer principle for valuations. After some lengthy computation, it was found in \cite{bernig_fu_solanes} that the global kinematic formulas in the three cases are also identical.
A similar phenomenon has been observed in the spheres  $S^6,S^7$ under the respective actions of the groups $\mathrm{G}_2$ and $\mathrm{Spin}(7)$ (cf. \cite{solanes_wannerer}).

A good and conceptual explanation of these facts is unfortunately missing. A promising approach to understand better this phenomenon is to study the same question in the quaternionic case, and here the first non-trivial case is that of the family of quaternionic planes (i.e. $\mathbb{H}^2$, $\mathbb{H}P^2$, $\mathbb{H}H^2$). The present paper is a first step in this direction.

\subsection{More specific background}

Let us recall some fundamental results in the theory of valuations. The first one is McMullen's decomposition, which states that the space $\Val$ of all continuous and translation-invariant valuations on an $n$-dimensional real vector space $V$ decomposes as
\begin{equation} \label{eq_mcmullen}
\Val=\bigoplus_{\substack{k=0,\ldots,n\\ \epsilon=\pm}} \Val_k^\epsilon,
\end{equation} 
where $\Val_k^\epsilon$ is the subspace of $k$-homogeneous valuations of parity $\epsilon$ (i.e. those $\phi$ satisfying $\phi(tK)=t^k\phi(K), t>0$ and $\phi(-K)=\epsilon \phi(K)$). Since the group $\spsp$ contains $-\mathrm{Id}$, we will only be interested in even valuations. The spaces $\Val_0$ and $\Val_n$ are $1$-dimensional and spanned by the Euler-characteristic $\chi$ and the Lebesgue measure $\vol$ respectively. Given $\phi \in \Val$, we let $\phi_n \in \Val_n \cong \C$ be the $n$-th homogeneous component of $\phi$. 

An even valuation of degree $k$ can be described in terms of its Klain function as follows.  Let $E \in \Gr_k$ be a $k$-dimensional subspace and $\phi \in \Val_k^+$. Then the restriction $\phi|_E$ is a multiple of the Lebesgue measure, and we denote by $\Kl_\phi(E)$ the proportionality factor. The corresponding function $\Kl_\phi \in C(\Gr_k)$ is called the Klain function of $\phi$. It uniquely determines $\phi$ by a theorem of Klain \cite{klain00}. 

Alesker introduced a graded product on a certain dense subspace $\Val^{sm}$ of smooth valuations (see Definition \ref{def_smoothness}). If $G$ is a subgroup of $\mathrm{SO}(n)$ acting transitively on the unit sphere, then $\Val^G \subset \Val^{sm}$ and $\Val^G$ is a finite-dimensional algebra. 

Alesker proved that the pairing 
\begin{align*}
 \Val^{sm} \times \Val^{sm} & \to \mathbb{C}\\
 (\phi,\psi) & \mapsto (\phi \cdot \psi)_n
\end{align*}
is perfect, which means that the induced map 
\begin{displaymath}
 \mathrm{pd}:\Val^{sm} \to \Val^{sm,*}
\end{displaymath}
is injective and has dense image \cite{alesker04_product}. The restriction $\mathrm{pd}_G:\Val^G \to \Val^{G*}$ is therefore an isomorphism.

The knowledge of the algebra structure is equivalent to the knowledge of the kinematic formulas as follows. Let $m_G:\Val^G \otimes \Val^G \to \Val^G$ denote the product and $m_G^*:\Val^{G*} \to \Val^{G*} \otimes \Val^{G*}$ its adjoint. The kinematic operator $k=k_G:\Val^G \to \Val^G \otimes \Val^G$ is defined by 
\begin{displaymath}
k_G(\phi)(K,L):=\int_{\bar G} \phi(K \cap \bar g L) d\bar g, 
\end{displaymath}
where $\bar G$ is the semi-direct product of $G$ (endowed with Haar probability measure) and the translation group of $V$ (endowed with the Lebesgue measure). 

In \cite{bernig_fu06} it was shown that the following diagram commutes
\begin{displaymath}
 \xymatrix{\Val^G \ar[r]^-{k_G} \ar[d]_{\mathrm{pd}_G} & \Val^G \otimes \Val^G
\ar[d]^{\mathrm{pd}_G \otimes \mathrm{pd}_G} \\
\Val^{G*} \ar[r]^-{m_G^*} & \Val^{G*} \otimes \Val^{G*}.}
\end{displaymath}
In particular, the principal kinematic formula 
\begin{displaymath}
k_G(\chi) \in \Val^G \otimes \Val^G=\mathrm{Hom}(\Val^{G*},\Val^G)
\end{displaymath}
is inverse to $\mathrm{pd}_G \in \mathrm{Hom}(\Val^G,\Val^{G*})$ (which was observed earlier in \cite{fu06}).

Alesker has also introduced an operation $\mathbb{F}:\Val^{sm} \to \Val^{sm}$, called the Alesker-Fourier transform. In the even case, it is characterized by the property 
\begin{displaymath}
\Kl_{\mathbb F \phi}=\Kl_{\phi} \circ \perp, 
\end{displaymath}
where $\perp:\Gr_k \to \Gr_{n-k}$ is the orthogonal complement map.

Let us next recall some results from \cite{bernig_solanes}.  The group of all $\mathbb{H}$-linear automorphisms on the right $\mathbb{H}$-vector space $\mathbb{H}^2$ is denoted by $\mathrm{GL}(2,\mathbb{H})$. The subgroup of $\GL(2,\h)$ of all elements preserving the standard quaternionic hermitian form on $\mathbb{H}^2$ is called the {\it compact
symplectic group} and denoted by $\Sp(2)$. This group acts by left matrix multiplication on $\mathbb{H}^2$, while the group of unit quaternions $\Sp(1)$ acts by scalars on the right. These two actions clearly commute and induce an action of $\spsp:=\Sp(2) \times \Sp(1)/\{\pm (\mathrm{Id},1)\}$ on $\mathbb{H}^2$. 

Let $\Val_k^{\spsp}$ be the space of $k$-homogeneous, continuous, translation-invariant and $\spsp$-invariant valuations. Let us describe a basis of $\Val_k^{\spsp}$ for $2 \leq k \leq 4$ (the cases $k=0,1$ are trivial, and the Alesker-Fourier transform induces an isomorphism  $\Val_{8-k}^{\spsp} \cong \Val_k^{\spsp}$, so it is enough to restrict to these values of $k$). 

It was shown in \cite{bernig_solanes} that every $\spsp$-orbit on the Grassmann manifold $\Gr_k(\mathbb{H}^2), 2 \leq k \leq 4,$ contains a plane of the form 
\begin{align*}
&\spann\{(\cos\theta_1,\sin\theta_1),(\cos\theta_2,\sin\theta_2)\mathbf i\} & k=2,\\
&\spann\{(\cos\theta_1,\sin\theta_1),(\cos\theta_2,\sin\theta_2)\mathbf i,(\cos\theta_3,\sin\theta_3)\mathbf j\} &
k=3,\\
&\spann\{(\cos\theta_1,\sin\theta_1),(\cos\theta_2,\sin\theta_2)\mathbf i,(\cos\theta_3,\sin\theta_3)\mathbf
j,(\cos\theta_4,\sin\theta_4)\mathbf k\} & k=4,
\end{align*}
where $\theta_1,\ldots,\theta_4 \in [0,2\pi]$. Set $\lambda_{pq}=\cos(\theta_p-\theta_q)$. 

Define invariant functions $f_{k,i} \in C^\infty(\Gr_k(\mathbb{H}^2))$ by 
\begin{align*}
f_{k,0}(\lambda) & := 1, \quad k=0,\ldots,4\\
f_{2,1}(\lambda) & :=\lambda_{12}^2\\
f_{3,1}(\lambda) & :=\lambda_{12}^2+\lambda_{13}^2+\lambda_{23}^2\\ 
f_{3,2}(\lambda) & :=\lambda_{12}^2\lambda_{23}^2+\lambda_{13}^2\lambda_{23}^2+ \lambda_{12}^2 \lambda_{13}^2\\
f_{4,1}(\lambda) &
:=\lambda_{12}^2+\lambda_{13}^2+\lambda_{14}^2+\lambda_{23}^2+\lambda_{24}^2+\lambda_{34}^2\\
f_{4,2}(\lambda) & :=\lambda_{12}^2\lambda_{34}^2+\lambda_{13}^2\lambda_{24}^2+\lambda_{14}^2\lambda_{23}^2\\
f_{4,3}(\lambda) &
:= \lambda_{12}^2\lambda_{13}^2+\lambda_{12}^2\lambda_{14}^2+\lambda_{13}^2\lambda_{14}^2+\lambda_{12}^2\lambda_{
23 } ^2+\lambda_{12}^2\lambda_{24}^2+\lambda_{23}^2\lambda_{24}^2\\
& \quad +
\lambda_{13}^2\lambda_{23}^2+\lambda_{13}^2\lambda_{34}^2+\lambda_{23}^2\lambda_{34}^2+\lambda_{14}^2\lambda_{24}
^2+\lambda_{14}^2\lambda_{34}^2+\lambda_{24}^2\lambda_{34}^2\\
f_{4,4}(\lambda) & :=2\lambda_{12}\lambda_{13}\lambda_{23}^2\lambda_{24}\lambda_{34}
+2\lambda_{12}\lambda_{13}\lambda_{14}^2\lambda_{24}\lambda_{34}
+2\lambda_{12}\lambda_{23}\lambda_{13}^2\lambda_{14}\lambda_{34}\\
& \quad +2\lambda_{12}\lambda_{23}\lambda_{24}^2\lambda_{14}\lambda_{34}
+2\lambda_{24}\lambda_{23}\lambda_{12}^2\lambda_{14}\lambda_{13}
+2\lambda_{24}\lambda_{23}\lambda_{34}^2\lambda_{14}\lambda_{13}\\
& \quad
+3( \lambda_{12}^2\lambda_{13}^2\lambda_{14}^2+\lambda_{12}^2\lambda_{23}^2\lambda_{24}^2+\lambda_{13}^2\lambda_{
23} ^2\lambda_{34}^2+\lambda_{14}^2\lambda_{ 24 }^2\lambda_{34}^2).
\end{align*} 

In \cite{bernig_solanes} it was shown that there exist valuations $\varphi_{k,i}\in\Val_k^{\spsp}$ such that $\Kl_{\varphi_{k,i}}=\Kl_{\varphi_{8-k,i}}=f_{k,i}$ for $0\leq k\leq 4$.

\subsection{Results of the present paper}

Our first main theorem is the principal kinematic formula $k(\chi)=k_{\spsp}(\chi)$ on the quaternionic plane. We state it here for the basis $\varphi_{k,i}$ only, but we will give all necessary information to rewrite it in the other bases which will be introduced below. 

The symmetric product of two elements $\phi,\psi \in \Val^{\spsp}$ will be denoted by $\phi \odot \psi:=\frac12(\phi \otimes \psi+\psi \otimes \phi)$.
 
\begin{MainTheorem} \label{mainthm_kinform}
The principal kinematic formula in the quaternionic plane $\mathbb H^2$ with respect to the group $\overline\spsp$ is given by 
\begin{align*}
k(\chi) & =\int_{\overline{\spsp}} \chi(\cdot\cap \bar g\cdot)d\bar g\\
& = 2\varphi_{0,0}\odot \varphi_{8,0}
+{\frac {64}{35\,\pi }\,\varphi_{{1,0}}\odot \varphi_{{7,0}}}
+{\frac {5}{16}\,\varphi_{{2,0}}\odot \varphi_{{6,0}}}
-\frac{1}{16}\,\varphi_{{2,0}}\odot \varphi_{{6,1}}\\&
-\frac{1}{16}\,\varphi_{{2,1}}\odot \varphi_{{6,0}}
+{\frac {7}{48}\,\varphi_{{2,1}}\odot \varphi_{{6,1}}}
+{\frac {248}{315\,\pi }\,\varphi_{{3,0}}\odot \varphi_{{5,0}}}
-{\frac {104}{945\,\pi }\,\varphi_{{3,0}}\odot \varphi_{{5,1}}}\\&
-{\frac {16}{189\,\pi }\,\varphi_{{3,0}}\odot \varphi_{{5,2}}}
-{\frac {104}{945\,\pi }\,\varphi_{{3,1}}\odot \varphi_{{5,0}}}
+{\frac {152}{2835\,\pi }\,\varphi_{{3,1}}\odot \varphi_{{5,1}}}
+{\frac {272}{2835\,\pi }\,\varphi_{{3,1}}\odot \varphi_{{5,2}}}\\&
-{\frac {16}{189\,\pi }\,\varphi_{{3,2}}\odot \varphi_{{5,0}}}
+{\frac {272}{2835\,\pi }\,\varphi_{{3,2}}\odot \varphi_{{5,1}}}
-{\frac {256}{2835\,\pi }\,\varphi_{{3,2}}\odot \varphi_{{5,2}}}
+{\frac {293}{1920}\,{\varphi_{{4,0}}}\odot\varphi_{{4,0}}}\\&
-{\frac {143}{2880}\,\varphi_{{4,0}}\odot \varphi_{{4,1}}}
+{\frac {103}{2880}\,\varphi_{{4,0}}\odot \varphi_{{4,2}}}
-{\frac {5}{1152}\,\varphi_{{4,0}}\odot \varphi_{{4,3}}}
-{\frac {7}{384}\,\varphi_{{4,0}}\odot \varphi_{{4,4}}}\\&
+{\frac {317}{17280}\,{\varphi_{{4,1}}}\odot \varphi_{{4,1}}}
-{\frac {469}{8640}\,\varphi_{{4,1}}\odot \varphi_{{4,2}}}
-{\frac {293}{17280}\,\varphi_{{4,1}}\odot \varphi_{{4,3}}}
+{\frac {113}{5760}\,\varphi_{{4,1}}\odot \varphi_{{4,4}}}\\&
+{\frac {797}{17280}\, \varphi_{{4,2}}\odot \varphi_{{4,2}}}
+{\frac {637}{17280}\,\varphi_{{4,2}}\odot \varphi_{{4,3}}}
-{\frac {97}{5760}\,\varphi_{{4,2}}\odot \varphi_{{4,4}}}
+{\frac {701}{69120}\, \varphi_{{4,3}}\odot \varphi_{{4,3}}}\\&
-{\frac {161}{11520}\,\varphi_{{4,3}}\odot \varphi_{{4,4}}}
+{\frac {7}{2560}\, \varphi_{{4,4}}\odot\varphi_{{4,4}} }.
\end{align*}
\end{MainTheorem}

Let us introduce a more geometric basis as follows. Let  $\mathbb{H}P^1 \subset \Gr_4$ be the quaternionic $1$-Grassmann manifold, i.e. the space of quaternionic lines inside $\mathbb{H}^2$. We endow it with the invariant probability measure.  Let $\mu_i$ denote the $i$-th intrinsic volume. Define $t,s,v,u\in\Val^{\spsp}$ by
\begin{align*}
 t(K) & :=  \frac{35}{12\pi} \int_{\mathbb{H}P^1} \mu_1(\pi_EK)dE\\
 s(K) & := \int_{\mathbb{H}P^1} \mu_2(\pi_EK)dE\\ 
 v(K) & := \int_{\mathbb{H}P^1} \mu_3(\pi_EK)dE\\ 
 u(K) & := \int_{\mathbb{H}P^1} \mu_4(\pi_EK)dE. 
\end{align*}

We may express these elements in terms of the basis $\varphi_{k,i}$ as follows
\begin{align*}
 t & = \frac{2}{\pi}\varphi_{1,0}\\
 s & = \frac{3}{8}\varphi_{2,0}+\frac{1}{8}\varphi_{2,2}\\ 
 v & = \frac{8}{105}\varphi_{3,0}+\frac{8}{63}\varphi_{3,1}-\frac{8}{315}\varphi_{3,2}\\ 
 u & =-\frac{3}{16}\varphi_{4,0}+\frac{11}{80}\varphi_{4,1}+\frac{1}{80}\varphi_{4,2}-\frac{7}{160}\varphi_{4,3}+\frac{1}{160}\varphi_{4,4}. 
\end{align*}

\begin{MainTheorem}
The valuations $t,s,v,u$ generate the algebra $\Val^{\spsp}$. A basis of $\Val^{\spsp}$ (as a vector space) is given by \begin{displaymath}
1,t,t^2,s,t^3,ts, v,t^4,t^2s,s^2,tv,u,t^5,t^2v,t^3s,t^6,t^4s,t^7,t^8.
\end{displaymath}
\end{MainTheorem}

The algebra structure of $\Val^{\spsp}$ in terms of these generators will be described in Section \ref{sec_algebra_structure}; the kinematic formulas in this basis are then an easy consequence.

We end the introduction with some remarks on higher dimensions. The dimension of the spaces $\Val_k^{\Sp(n) \Sp(1)}$ was computed in full generality in \cite{bernig_qig}. These numbers behave irregularly. Even in the case $n=3$, a Hadwiger type theorem is unknown. It may still be possible to work out in general the invariant differential forms as we do it below. This may be an entry point to the general theory.

Taking into account the hermitian case, where $\Val^{\mathrm{U}(n)}$ is generated for all $n$ by two valuations $t$ and $s$, a tempting conjecture would be that the valuations $t,s,u,v$ (which can be defined for general $n$ as above) generate $\Val^{\Sp(n)\Sp(1)}$. However, this turns out to be false for $n \geq 3$, since $\dim \Val_3^{\Sp(n)\Sp(1)}=4$, while there are only three independent valuations of degree $3$ in the algebra generated by $t,s,v,u$ (namely $t^3,ts,v$).

\subsection{Plan of the paper}

In Section \ref{sec_multipliers}, we prove a result of independent interest on the harmonic analysis of valuations. More precisely, we compute the multipliers of some natural operators on valuations. These operators are compositions of the derivation operator, the product with the first intrinsic volume, and the Alesker-Fourier transform. The main ingredients are an explicit description of highest weight vectors in the decomposition of $L^2(\Gr_k)$ into irreducible $\mathrm{SO}(n)$-modules (due to Strichartz) and the computation of the multipliers of the Radon transform (due to Grinberg). 

In Section \ref{sec_symmetry} we introduce a new operator $I^*$ on the space of translation-invariant smooth curvature measures.

Starting with Section \ref{sec_module_decomposition}, we consider the quaternionic plane. We will describe the space of invariant differential forms on the sphere bundle. In Section \ref{sec_rumin} we give an algorithm to compute the Rumin differential and state the result of this computation. 

By looking at the eigenvalues of the composition of $I^*$ with powers of some derivation operator, and using the multipliers from the previous section, we will in Section \ref{sec_klain_functions} compute the Klain functions of the valuations represented by the forms. The Alesker-Fourier transform can then be easily deduced.  

In Section \ref{sec_algebra_structure} we prove the  principal kinematic formula, i.e. we prove Theorem \ref{mainthm_kinform}.  We also determine the algebra structure of $\Val^{\spsp}$, which allows to compute the rest of kinematic formulas. The key point here is that the convolution product from \cite{bernig_fu06} can be computed by some linear operations on forms (once the Rumin differentials are known), and the product can be deduced from it once the Alesker-Fourier transform is known.

\subsection*{Acknowledgments}
We thank the anonymous referees for useful comments.
\section{Multipliers of some natural operators on valuations}
\label{sec_multipliers}

We begin by recalling the definitions of smooth translation-invariant valuations and curvature measures and several algebraic structures on the space of smooth translation-invariant valuations. We will then prove some results from the harmonic analysis of translation-invariant valuations which are of independent interest. 

\begin{Definition} \label{def_smoothness}
Let $V$ be a euclidean vector space of dimension $n$. A translation-invariant valuation $\mu$ is called smooth if it is of the form
 \begin{displaymath}
  \mu(K)=\int_{N(K)} \omega+c \vol(K),
 \end{displaymath}
 where $\omega$ is a translation-invariant differential form of degree $n-1$ on the sphere bundle $SV=V \times S^{n-1}$, $N(K)$ is the normal cycle of $K$ and $c \in \C$. Similarly, a translation-invariant curvature measure $\Phi$ is smooth if
 \begin{displaymath}
  \Phi(K,U)=\int_{N(K) \cap \pi^{-1}U} \omega+c \vol(K \cap U),
 \end{displaymath}
where $U$ is a Borel subset of $V$ and $\pi:SV \to V$ is the natural projection. 
The space of smooth translation-invariant valuations is denoted by $\Val^{sm}$, and the space of smooth translation-invariant curvature measures is denoted by $\mathrm{Curv}$. Taking $U:=V$ induces a map $\glob:\Curv \to \Val^{sm}$, called globalization.
\end{Definition}

The form $\omega$ in the definition of a smooth valuation or a smooth curvature measure is not unique. A consequence of the main theorem in \cite{bernig_broecker07} is that $(\omega,c)$ induces the trivial valuation if and only if $D\omega=0$ and $c=0$, where $D$ is the Rumin operator (see Section \ref{sec_rumin}). The pair $(\omega,c)$ induces the trivial curvature measure if and only if $c=0$ and $\omega$ is in the ideal generated by $\alpha$,  the canonical $1$-form on $SV$, and its differential $d\alpha$.

Alesker has introduced a natural product structure on the space of smooth valuations which is compatible with McMullen's decomposition \eqref{eq_mcmullen}. The Alesker product of $\phi,\psi$ will be denoted by $\phi \cdot \psi$. In the special case where $\phi(K)=\vol_n(K+A), \psi(K)=\vol_n(K+B)$ with $A,B$ smooth convex bodies with positive curvature, the product is $\phi \cdot \psi(K)=\vol_{2n}(\Delta K+A \times B)$, where $\Delta:V \to V \times V$ is the diagonal embedding.

A related construction of convolution was introduced in \cite{bernig_fu06}. For smooth valuations $\phi(K)=\vol_n(K+A), \psi(K)=\vol_n(K+B)$ as above, the convolution satisfies $\phi * \psi(K)=\vol_n(K+A+B)$. 

Furthermore, there is a Fourier type transform $\mathbb F:\Val^{sm} \to \Val^{sm}$ such that $\mathbb{F}(\phi \cdot \psi)=\mathbb{F}\phi * \mathbb{F}\psi$ and $\mathbb{F}^2 \phi(K)=\phi(-K)$. The definition and construction in the odd case are rather involved, and for the purposes of this paper we may restrict to the even case. If $\phi \in \Val_k^{sm,+}$, then $\mathbb{F}\phi \in \Val_{n-k}^{sm,+}$ is characterized by the property 
\begin{equation} \label{eq_fourier}
 \Kl_{\mathbb F \phi}(E)=\Kl_\phi(E^\perp),\quad E \in \Gr_{n-k}(V).
\end{equation}

Given that every smooth translation-invariant valuation may be represented by a pair $(\omega,c)$, with $\omega$ a differential form on the sphere bundle, a natural question is how the algebraic operations product, convolution and Alesker-Fourier transform can be expressed in terms of forms. For the product, a rather complicated formula is given in \cite{alesker_bernig}, compare also \cite{fu_alesker_product} for a geometric explanation of this formula. For the Alesker-Fourier transform, a description on the level of forms seems to be unknown. However, the convolution product admits a rather simple description in terms of forms, which makes its computation in many cases possible. For further reference, we recall this construction from \cite{bernig_fu06}.

Let $\phi,\psi \in \Val^{sm}$ be represented by pairs $(\omega_1,c_1),(\omega_2,c_2)$. Then the convolution $\phi * \psi$ is represented by the pair 
\begin{equation}\label{eq_convolution}
 (*_1^{-1}(*_1 \omega_1 \wedge *_1D\omega_2)+c_1\omega_2+c_2\omega_1, c_1c_2).
\end{equation}
Here $*_1$ is the part of the Hodge star operator acting on the first factor of $SV=V \times S^{n-1}$ (endowed with some sign, cf. \cite[Eq. (36)]{bernig_fu06} and Section \ref{sec_algebra_structure}).  

\begin{Definition}
 Multiplication by the first intrinsic volume $\mu_1$ is an operator denoted by 
 \begin{displaymath}
  L:\Val^{sm} \to \Val^{sm}. 
 \end{displaymath}
The derivation operator 
\begin{displaymath}
  \Lambda:\Val^{sm} \to \Val^{sm}
 \end{displaymath}
is defined by 
\begin{displaymath}
 \Lambda \phi(K)=\left.\frac{d}{dt}\right|_{t=0} \phi(K+tB),
\end{displaymath}
where $B$ is the unit ball. Equivalently, $\Lambda \phi=2 \phi * \mu_{n-1}$. We then have 
\begin{equation} \label{eq_relation_lambda_l}
\Lambda \circ \mathbb{F}=2 \mathbb{F} \circ L. 
\end{equation}
\end{Definition}

On the level of forms, the operator $\Lambda$ is given by $\omega \mapsto i_TD\omega$, where $T$ is the Reeb vector field on the contact manifold $SV$ and $D$ is the Rumin operator. A similar easy expression for $L$ is not known. 

Let us now describe the main result from \cite{alesker_bernig_schuster}. Let $\mathrm{SO}(n)$ be the rotation group of $V$. It acts naturally on $\Val=\Val(V)$ by $g \mu(K):=\mu(g^{-1}K)$. The spaces $\Val_k$ are invariant under this action. The decomposition of $\Val_k$ as a direct sum of irreducible $\mathrm{SO}(n)$-modules is as follows. 

Recall that irreducible $\mathrm{SO}(n)$-modules are indexed by their highest weights, and the possible highest weights are given by sequences $\lambda=(\lambda_1,\ldots,\lambda_{\lfloor \frac{n}{2} \rfloor})$ of integers with $\lambda_1 \geq \lambda_2 \geq \ldots \geq \lambda_{n/2-1} \geq |\lambda_{n/2}| \geq 0$ if $n$ is even; and with $\lambda_1 \geq \lambda_2 \geq \ldots \geq \lambda_{(n-1)/2} \geq 0$ if $n$ is odd. We write $\Gamma_\lambda$ for the corresponding module. 

\begin{Theorem}[Alesker-Bernig-Schuster, \cite{alesker_bernig_schuster}] \label{thm_alesker_bernig_schuster}
The  $\mathrm{SO}(n)$-module $\Val_k$ is multiplicity-free. The irreducible $\mathrm{SO}(n)$-modules $\Gamma_\lambda$ entering $\Val_k$ are characterized by the following three conditions: 
\begin{enumerate}
 \item $|\lambda_i| \neq 1$ for all $i$.
 \item $|\lambda_2| \leq 2$.
 \item $\lambda_i=0$ for $i>\min\{k,n-k\}$.
\end{enumerate}
Even valuations correspond to those $\Gamma_\lambda$ which satisfy in addition $\lambda_1 \equiv 0 \mod 2$.
\end{Theorem}

The isotypical component corresponding to $\Gamma_\lambda$ in $\Val_k$ will be denoted by $\Val_k[\Gamma_\lambda]$. More generally, if $W$ is a representation of $\mathrm{SO}(n)$, the isotypical component corresponding to $\Gamma_\lambda$ will be denoted by $W[\Gamma_\lambda]$.   

\begin{Proposition} \label{prop_multiplier_fourier}
Let  $V$ be a euclidean vector space of even dimension $n=2\nu$. The map $F \mapsto \hat F:=F \circ \perp$ acts
on $L^2(\Gr_\nu V)[\Gamma_\lambda]$ as multiplication by the scalar $(-1)^\frac{|\lambda|}{2}$ where $|\lambda|=\lambda_1+\ldots+\lambda_\nu$. 
\end{Proposition}

\proof

By \cite[Section 5]{strichartz75}, $L^2(\Gr_\nu V)$ is a multiplicity free $\mathrm{SO}(n)$-module. Schur's lemma implies that $\hat F=c_{\lambda} F$ for all $F \in L^2(\Gr_\nu V)[\Gamma_\lambda]$, where $c_\lambda$ is independent of
$F$. To determine $c_\lambda$ explicitly, it is
enough to compute $\hat F$ for one special function $F$. 

Let $\lambda=(\lambda_1,\ldots,\lambda_\nu)$ be a highest weight appearing in $L^2(\Gr_\nu  V)$. 
A highest weight vector in $L^2(\Gr_\nu V)[\Gamma_\lambda]$ was
constructed in \cite{strichartz75} as follows, compare also \cite{grinberg86}.
 
Let us fix an orthonormal basis $e_1,\ldots,e_n$ of $V$. 
A subspace $E \in \Gr_\nu V$ can be described by a Stiefel matrix $X$, whose columns represent an orthonormal basis of
$E$. 

For $1 \leq l \leq \nu$ let $A(l)$ be the $l\times\nu$ matrix whose $j$-th row is given by the
$(2j-1)$-th row of $X$ plus $\sqrt{-1}$ times the $2j$-th row of $X$. Define 
\begin{align*}
F_l & :=\det (A(l) \cdot A(l)^t)\\
G_+ & :=\det A(\nu), G_-:=\overline{\det A(\nu)}. 
\end{align*}

Let $r_1,\ldots,r_{\nu-1}, s$ non-negative integers. Then the function
\begin{displaymath}
F:=G_\pm^s \prod_{l=1}^{\nu-1} F_l^{r_l}
\end{displaymath}
is a highest weight vector corresponding to the highest weight 
\begin{displaymath}
 \lambda_j=\left\{ \begin{array}{c c} s+\sum_{l=j}^{\nu-1} 2r_l & j \leq \nu-1\\
              \pm s, & j=\nu.
             \end{array}
\right.
\end{displaymath}

Let $E_0 \in \Gr_\nu  V$ be the span of the vectors $e_{2j-1}, j=1,\ldots,\nu$. Then 
\begin{displaymath}
 A(l)=\left( \begin{array}{c c} \mathrm{Id}_l & 0 \end{array} \right),
\end{displaymath}
and hence $F_l(E_0)=1, G_\pm(E_0)=1$ and therefore $F(E_0)=1$.  

The orthogonal complement $E^\perp_0$ is the span of the
vectors $e_{2j}, j=1,\ldots,\nu$ and we have  
\begin{displaymath}
 A(l)=\left( \begin{array}{c c} \sqrt{-1} \mathrm{Id}_l & 0 \end{array} \right),
\end{displaymath}
and hence $F_l(E_0^\perp)=(-1)^l, G_\pm(E_0^\perp)=(\pm \sqrt{-1})^\nu$.

It follows that 
\begin{displaymath}
 \hat F(E_0)=F(E_0^\perp)=(-1)^{\sum_l l r_l}(\pm \sqrt{-1})^{s \nu} F(E_0).
\end{displaymath}
If $\Gamma_\lambda$ appears in $L^2(\Gr_\nu V)$, then $s \equiv 0 \mod 2$ (compare \cite{strichartz75}), hence 
\begin{displaymath}
 c_\lambda=(-1)^{\sum_l l r_l+\frac{s\nu}{2} }.
\end{displaymath}

The statement follows from  
\begin{displaymath}
|\lambda| = \sum_{j=1}^\nu \lambda_j \equiv  2 \sum_{j=1}^{\nu-1} \sum_{l=j}^{\nu-1} r_l+s\nu=2 \sum_{l=1}^{\nu-1} l r_l +s\nu \mod 4.
\end{displaymath}
\endproof

\begin{Corollary}
 Let  $V$ be a euclidean vector space of even dimension $n=2\nu$. Then the Alesker-Fourier transform acts
on $\Val^+_\nu[\Gamma_\lambda]$ as multiplication by the scalar $(-1)^\frac{|\lambda|}{2}$ where $|\lambda|=\lambda_1+\ldots+\lambda_\nu$. 
\end{Corollary}

\proof
$\Val^+_\nu$ can be considered as a submodule of $L^2(\Gr_\nu)$ via the Klain map. The Alesker-Fourier transform corresponds to the natural involution on $\Gr_\nu$ given by 
\begin{displaymath}
 \hat f(E):=f(E^\perp),
\end{displaymath}
which implies the statement. 
\endproof

Next we will describe the action of $\Lambda\circ L$ and $L\circ\Lambda$ on $\Val_k$. First, we  fix some notation and recall the definitions of the cosine and the Radon transforms.

The volume of the $n$-dimensional unit ball is denoted by $\omega_n$. The flag coefficient is defined by 
\begin{displaymath}
 \flag{n}{k}:=\binom{n}{k} \frac{\omega_n}{\omega_k \omega_{n-k}}.
\end{displaymath}

The cosine transform is the map
\begin{displaymath}
 T_{l,k} : C^\infty(\Gr_k(V)) \to C^\infty(\Gr_l(V))
\end{displaymath}
such that 
\begin{displaymath}
T_{l,k}f(F)=\int_{\Gr_k(V)} |\cos(E,F)| f(E) dE, \quad F \in \Gr_lV.
\end{displaymath}
Here $\Gr_k V$ is endowed with the invariant probability measure and $|\cos(E,F)|$ is the cosine of the angle between $E$ and $F$. We refer to  \cite{alesker_bernstein04, olafsson_pasquale} for more information on this important integral transform. 

The Radon transform is the map  
\begin{displaymath}
 R_{l,k} : C^\infty(\Gr_k(V)) \to C^\infty(\Gr_l(V))
\end{displaymath}
such that 
\begin{displaymath}
R_{l,k}f(F)= \begin{cases} \int_{E \subset F} f(E) dE & k \leq l\\ \int_{E \supset F} f(E) dE & k \geq l \end{cases},\quad E \in \Gr_k(V), F \in \Gr_l(V),
\end{displaymath}
where $dE$ denotes the invariant probability measure. 

It is well known that 
\begin{displaymath}
 T_{l,k} = \left\{\begin{array}{c c} T_{l,l} \circ R_{l,k} & k>l\\ R_{l,k} \circ T_{k,k} & k<l.\end{array}\right.
\end{displaymath}

\begin{Proposition} \label{prop_maps_as_radon}
 The map $\Lambda \circ L$ acts as $2 \flag{n-k}{1} \flag{k+1}{1} R_{k,k+1} \circ R_{k+1,k}$ on the image of $\Val_k^{sm,+}$
in $C^\infty(\Gr_kV)$. The map $L \circ \Lambda$ acts as $2 \flag{k}{1}\flag{n-k+1}{1} R_{k,k-1} \circ R_{k-1,k}$. 
\end{Proposition}

\proof
By \cite{alesker03_un}, the map $\Lambda^j$ corresponds to
the map 
\begin{displaymath}
T_{k-j,k} \circ T_{k,k}^{-1} :C^\infty(\Gr_kV) \to
C^\infty(\Gr_{k-j}V) 
\end{displaymath}
(up to a normalizing constant). Note that $T_{k,k}^{-1}$ is well-defined on the image of $\Val_k^{sm,+}$  (cf. \cite{alesker_bernstein04}).  

 By \cite[Lemma 2.5]{alesker04}, $L^j$ corresponds to 
\begin{displaymath}
T_{k+j,k+j} \circ R_{k+j,k} \circ T_{k,k}^{-1}:C^\infty(\Gr_kV) \to
C^\infty(\Gr_{k+j}V). 
\end{displaymath}

Hence $\Lambda \circ L$ acts as 
\begin{align*}
(T_{k,k+1} \circ T_{k+1,k+1}^{-1}) \circ (T_{k+1,k+1} \circ R_{k+1,k} \circ T_{k,k}^{-1}) & = T_{k,k+1} \circ R_{k+1,k}
\circ T_{k,k}^{-1}\\
& = T_{k,k} \circ R_{k,k+1} \circ R_{k+1,k} \circ T_{k,k}^{-1}\\
& = R_{k,k+1} \circ R_{k+1,k}.
\end{align*}

Similarly, $L \circ \Lambda$ acts (up to a normalizing constant) as 
\begin{align*}
 (T_{k,k} \circ R_{k,k-1} \circ T_{k-1,k-1}^{-1}) \circ (T_{k-1,k} \circ T_{k,k}^{-1})=R_{k,k-1} \circ R_{k-1,k}.
\end{align*}
To fix the constants, we consider $\phi:=\mu_k$, whose image in $C(\Gr_kV)$ is the function which is constant $1$. We
have $R_{k,k+1} \circ R_{k+1,k}1=1$, $R_{k,k-1} \circ R_{k-1,k}1=1$. On the other hand $\Lambda \circ L \mu_k=2 \flag{n-k}{1} \flag{k+1}{1} \mu_k$ and $L \circ \Lambda \mu_k=2 \flag{n-k+1}{1} \flag{k}{1} \mu_k$.  
\endproof

\begin{Proposition} \label{prop_actions_lambda_l}
Let $2k+1 \leq n$ and let $\Gamma_\lambda$ appear in $\Val_k^+$. Let $a:=|\lambda_1|$ and let $b$ be the depth of $\lambda$ (i.e. $\lambda_b \neq 0, \lambda_{b+1}=0$), and set $b':=\max\{b,1\}$. Then $L \circ \Lambda$ acts on
$\Val_k^+[\Gamma_\lambda]$ by the scalar
\begin{displaymath}
\frac{\pi^2}{2\flag{a+k-1}{1}\flag{a+n-k}{1}}
(k-b')(n+1-k-b').
\end{displaymath}
$\Lambda \circ L$ acts on $\Val_k^+[\Gamma_\lambda]$ by the scalar
\begin{displaymath}
\frac{\pi^2}{2\flag{a+k}{1}\flag{a+n-k-1}{1}}
(k+1-b')(n-k-b').
\end{displaymath}
\end{Proposition}

\proof
Let $\Gamma_\lambda$ be an irreducible representation of $\mathrm{SO}(n)$ which appears in $L^2(\Gr_kV)$. By \cite[Section 5]{grinberg86} $R_{k,k+1} \circ
R_{k+1,k}$ acts as the following scalar on $L^2(\Gr_kV)[\Gamma_\lambda]$:

\begin{equation} \label{eq_mult_radon}
\frac{\Gamma\left(\frac{k+1}{2}\right)\Gamma\left(\frac{n-k}{2}\right)}{\Gamma\left(\frac12\right) \Gamma\left(\frac{n-2k}{2}\right)} \prod_{j=1}^k
\frac{\Gamma\left(\frac{\lambda_j+k+1-j}{2}\right)\Gamma\left(\frac{\lambda_j+n-k-j}{2}\right)}{\Gamma\left(\frac{\lambda_j +k+2-j}{2}\right)
\Gamma\left(\frac{\lambda_j+n-j-k+1}{2}\right)}.
\end{equation}

Let $\Gamma_\lambda$ appear in $\Val_k^+$. Let us assume that $\lambda_1 \neq 0$, hence $b'=b$. If $n=2k$ is even and $\lambda_k \neq 0$, then $\Gamma_\lambda$ does not appear in $\Val_{k-1}^+[\Gamma_\lambda]$ and $\Val_{k+1}^+[\Gamma_\lambda]$, hence the multipliers for $L \circ \Lambda$ and $\Lambda \circ L$ have to be $0$, which is compatible with the displayed formulas. Otherwise, we have $\lambda_1=a, \lambda_2=\ldots=\lambda_b=2, \lambda_{b+1}=\ldots=\lambda_{\left\lfloor \frac{n}{2}\right\rfloor}=0$. The product thus splits as
\begin{align*}
 \prod_{j=1}^k &
\frac{\Gamma\left(\frac{\lambda_j+k+1-j}{2}\right) \Gamma\left(\frac{\lambda_j+n-k-j}{2}\right)}{\Gamma\left(\frac{\lambda_j +k+2-j}{2}\right)
\Gamma\left(\frac{\lambda_j+n-j-k+1}{2}\right)}  \\
& = \frac{\Gamma\left(\frac{a+k}{2}\right) \Gamma\left(\frac{a+n-k-1}{2}\right)}{\Gamma\left(\frac{a +k+1}{2}\right)
\Gamma\left(\frac{a+n-k}{2}\right)}   \prod_{j=2}^b
\frac{\Gamma\left(\frac{k+3-j}{2}\right)\Gamma\left(\frac{n-k-j+2}{2}\right)}{\Gamma\left(\frac{k+4-j}{2}\right)
\Gamma\left(\frac{n-j-k+3}{2}\right)} \\
& \quad \cdot \prod_{j=b+1}^k
\frac{\Gamma\left(\frac{k+1-j}{2}\right)\Gamma\left(\frac{n-k-j}{2}\right)}{\Gamma\left(\frac{k+2-j}{2}\right)
\Gamma\left(\frac{n-j-k+1}{2}\right)}\\
& = \frac{\Gamma\left(\frac{a+k}{2}\right) \Gamma\left(\frac{a+n-k-1}{2}\right)}{\Gamma\left(\frac{a +k+1}{2}\right)
\Gamma\left(\frac{a+n-k}{2}\right)} \cdot \frac{\Gamma\left(\frac{k-b+3}{2}\right) \Gamma\left(\frac{n-k-b+2}{2}\right)}{\Gamma\left(\frac{k+2}{2}\right)
\Gamma\left(\frac{n-k+1}{2}\right)} \cdot \frac{\Gamma\left(\frac{1}{2}\right) \Gamma\left(\frac{n-2k}{2}\right)}{\Gamma\left(\frac{k-b+1}{2}\right)
\Gamma\left(\frac{n-k-b}{2}\right)} \\
& = \frac{\Gamma\left(\frac{a+k}{2}\right) \Gamma\left(\frac{a+n-k-1}{2}\right)}{\Gamma\left(\frac{a +k+1}{2}\right)
\Gamma\left(\frac{a+n-k}{2}\right)} \cdot \frac{\Gamma\left(\frac{1}{2}\right) \Gamma\left(\frac{n-2k}{2}\right)}{\Gamma\left(\frac{k+2}{2}\right)
\Gamma\left(\frac{n-k+1}{2}\right)} \cdot \frac{(k-b+1)(n-k-b)}{4}
\end{align*}

Using 
\begin{displaymath}
 \flag{n}{1}=\sqrt{\pi} \frac{\Gamma\left(\frac{n+1}{2}\right)}{\Gamma\left(\frac{n}{2}\right)}, 
\end{displaymath}
we find that the constant \eqref{eq_mult_radon} is thus given by 
\begin{multline*}
 \frac{\Gamma\left(\frac{k+1}{2}\right)\Gamma\left(\frac{n-k}{2}\right)}{\Gamma\left(\frac{k+2}{2}\right)\Gamma\left(\frac{n-k+1}{2}\right)} \frac{\Gamma\left(\frac{a+k}{2}\right) \Gamma\left(\frac{a+n-k-1}{2}\right)}{\Gamma\left(\frac{a +k+1}{2}\right)
\Gamma\left(\frac{a+n-k}{2}\right)} \cdot   \frac{(k-b+1)(n-k-b)}{4}\\
=\frac{\pi^2}{\flag{k+1}{1} \flag{n-k}{1} \flag{a+k}{1} \flag{a+n-k-1}{1}} \cdot   \frac{(k-b+1)(n-k-b)}{4} 
\end{multline*}
Applying Proposition \ref{prop_maps_as_radon} yields the displayed multiplier for $\Lambda \circ L$. The case $\lambda_1=0$ can be checked directly. The computation for $L \circ \Lambda$ is similar. 
\endproof

\begin{Proposition} \label{prop_lambda_f}
 Let $k \leq \nu$, where $n=2\nu$ is even. Let $\Gamma_\lambda$ enter the decomposition of $\Val_k^+$. Let $a:=\lambda_1$, let $b$ be the depth of $\lambda$ (i.e. $\lambda_b \neq 0, \lambda_{b+1}=0$) and $b':=\max\{b,1\}$. Then
$\Lambda^{n-2k} \circ \mathbb{F}$ acts on $\Val_k^+[\Gamma_\lambda]$ by the scalar
\begin{displaymath}
 (-1)^\frac{|\lambda|}{2} \pi^{n-2k} \prod_{j=k}^{\nu-1} \flag{a+j}{1}^{-1}\flag{a+n-j-1}{1}^{-1} \cdot 
\frac{(n-k-b')!}{(k-b')!}. 
\end{displaymath}
\end{Proposition}

\proof
Using Propositions \ref{prop_multiplier_fourier} and \ref{prop_actions_lambda_l} and \eqref{eq_relation_lambda_l}, it follows that for $\psi \in \Val_k^+[\Gamma_\lambda]$ 
\begin{align*}
 \Lambda^{n-2k} \circ \mathbb{F}\psi & = \Lambda^{\nu-k} \circ \Lambda^{\nu-k} \circ \mathbb{F} \psi\\
& = 2^{\nu-k} \Lambda^{\nu-k} \circ \mathbb{F} \circ L^{\nu-k} \psi\\
& = 2^{\nu-k} (-1)^\frac{|\lambda|}{2} \Lambda^{\nu-k} \circ L^{\nu-k} \psi \\
& = 2^{\nu-k} (-1)^\frac{|\lambda|}{2} \prod_{j=k}^{\nu-1} \frac{\pi^2}{2\flag{a+j}{1}\flag{a+n-j-1}{1}} \cdot \\
& \quad \cdot (j+1-b')(n-j-b') \psi \\
& = (-1)^\frac{|\lambda|}{2} \pi^{n-2k} \prod_{j=k}^{\nu-1} \flag{a+j}{1}^{-1}\flag{a+n-j-1}{1}^{-1} \cdot \\
& \quad \cdot \frac{(n-k-b')!}{(k-b')!} \psi. 
\end{align*}
\endproof

For further reference, we state explicitly the values which will be needed in the rest of the paper. 

\begin{Corollary} \label{cor_multipliers}
We have the following multipliers. 
\begin{displaymath}
\begin{array}{c | c | c}
 \text{operator} & \text{module} & \text{scalar}\\ \hline
 \Lambda \circ L & \Val_2(\R^8)[\Gamma_{2,2,0,0}] & \frac{5 \pi}{6} \\
 \Lambda^2 \circ L^2 & \Val_2(\R^8)[\Gamma_{2,2,0,0}] &  \pi^2 \\
 \Lambda^3 \circ L^3 & \Val_2(\R^8)[\Gamma_{2,2,0,0}] &  \frac{6 \pi^3}{5}\\
 \Lambda^4 \circ L^4 & \Val_2(\R^8)[\Gamma_{2,2,0,0}] & \pi^4 \\
 \Lambda \circ L & \Val_3(\R^8)[\Gamma_{4,2,2,0}] & \frac{2\pi}{7} \\
 \Lambda^2 \circ L^2 & \Val_3(\R^8)[\Gamma_{4,2,2,0}] & \frac{4\pi^2}{49} \\
 \Lambda^2 \circ \mathbb F & \Val_3(\R^8)[\Gamma_{0,0,0,0}] & 8\pi\\
 \Lambda^2 \circ \mathbb F & \Val_3(\R^8)[\Gamma_{2,2,0,0}] & \frac{12}{5}\pi\\
 \Lambda^2 \circ \mathbb F & \Val_3(\R^8)[\Gamma_{4,2,2,0}] & \frac{4}{7}\pi\\
 \Lambda^4 \circ \mathbb F & \Val_2(\R^8)[\Gamma_{0,0,0,0}] & 60\pi^2\\ 
 \Lambda^4 \circ \mathbb F & \Val_2(\R^8)[\Gamma_{2,2,0,0}] & 4\pi^2\\ 
\end{array}
\end{displaymath}
\end{Corollary}

\proof
Some of these values follow directly from Proposition \ref{prop_actions_lambda_l} or Proposition \ref{prop_lambda_f}. For instance, Proposition \ref{prop_actions_lambda_l} implies that $\Lambda \circ L$ acts on $\Val_3(\R^8)[\Gamma_{4,2,2,0}]$ by the scalar $\frac27 \pi$. Since $\mathbb F$ acts trivially on $\Val_4(\R^8)[\Gamma_{4,2,2,0}]$ by Proposition \ref{prop_multiplier_fourier} we have $\Lambda \circ L=\mathbb F \circ \Lambda \circ L \circ \mathbb F=L \circ \mathbb F \circ \mathbb F \circ \Lambda=L \circ \Lambda$ on $\Val_4(\R^8)[\Gamma_{4,2,2,0}]$. It follows that $\Lambda^2 \circ L^2=(\Lambda \circ L)^2$ acts on $\Val_3(\R^8)[\Gamma_{4,2,2,0}]$ by the scalar $\left(\frac27 \pi\right)^2$. The other values are obtained in a similar way.
\endproof

\section{A symmetry operator on smooth translation-invariant curvature measures}
\label{sec_symmetry} 

Let $V$ be a euclidean vector space of dimension $n$ and $\Curv_k(V)$ the space of $k$-homogeneous smooth curvature measures. In this section, we construct an isomorphism of smooth translation-invariant curvature measures $I^*:\Curv_k(V) \cong \Curv_{n-k-1}(V), 0 \leq k \leq n-1$. The eigenvectors of certain operators constructed with the help of $I^*$ in the case $n=8$ will play a central role in later parts of this paper. 

Let $SV \cong V \times S^{n-1}$ be the sphere bundle over $V$. The tangent space at a given point $(p_0,v_0) \in SV$ splits naturally as 
\begin{displaymath}
 T_{p_0,v_0}SV= T_{p_0}V \oplus T_{v_0}S^{n-1}=\R v_0 \oplus v_0^\perp \oplus v_0^\perp.  
\end{displaymath}
The space $v_0^\perp \oplus v_0^\perp$ has a natural complex structure defined by 
\begin{displaymath}
 I(w_1,w_2):=(-w_2,w_1), \quad w_1,w_2 \in v_0^\perp. 
\end{displaymath}

Let $\Omega_v^{k,l}(SV)$ denote the space of vertical forms of bidegree $(k,l)$, i.e. multiples of the contact form $\alpha$ on $SV$ and $\Omega_h^{k,l}(SV):=\Omega^{k,l}(SV)/\Omega_v^{k,l}(SV)$ the space of horizontal forms. If $\omega \in \Omega_h^{k,l}(SV)$, then 
\begin{displaymath}
 \omega|_{(p_0,v_0)} \in \largewedge^k v_0^\perp \otimes \largewedge^l v_0^\perp.
\end{displaymath}

\begin{Proposition}
Define a map 
\begin{displaymath}
 I^*:\Omega_h^{k,l}(SV) \to \Omega^{l,k}_h(SV)
\end{displaymath}
by setting for each $(p_0,v_0) \in SV$,  
\begin{displaymath}
 (I^* \omega)|_{(p_0,v_0)}:=I^*(\omega|_{(p_0,v_0)}) \in \largewedge^l v_0^\perp \otimes \largewedge^k v_0^\perp.
\end{displaymath}
Then $I^*$ induces a $\mathrm{SO}(n)$-equivariant map
\begin{displaymath}
 I^*:\Curv_k(V) \to \Curv_{n-1-k}(V)
\end{displaymath}
on smooth translation-invariant curvature measures. It satisfies 
\begin{displaymath}
 (I^*)^2=(-1)^{n-1} \mathrm{Id}_{\Curv_k}.
\end{displaymath}
\end{Proposition}

\proof
We have a surjective map $\Omega_h^{k,n-k-1}(SV) \to \Curv_k(V)$, where $\omega$ is mapped to the curvature measure $(K,A) \mapsto \int_{N(K) \cap \pi^{-1}A} \omega$. Its kernel is given by multiples of the symplectic form $d\alpha$. It thus suffices to prove that multiples of the symplectic form are mapped to such. Since $I^*(\omega_1 \wedge \omega_2)=I^*\omega_1 \wedge I^*\omega_2$, the statement follows from the computation in coordinates at the point $(0,e_1)$
\begin{displaymath}
I^* d\alpha =I^* \sum_{j=2}^n dy_j \wedge dx_j= - \sum_{j=2}^n dx_j \wedge dy_j= \sum_{j=2}^n dy_j \wedge dx_j=d\alpha.  
\end{displaymath}
\endproof

We will need another operator on smooth curvature measures. 

\begin{Definition} \label{def_derivation_operator}
 Let $\mathcal{L}:\Curv_k(V) \to \Curv_{k-1}(V)$ be the operator which is given on the level of forms by $\omega \mapsto i_TD\omega$, where $D$ is the Rumin differential (see Section \ref{sec_rumin}) and $T$ is the Reeb vector field on $SV$. Since $D$ vanishes on multiples of $\alpha$ and $d\alpha$, this map is well-defined. 
\end{Definition}

Note the equation (cf. \cite[Lemma 2.5]{bernig_fu_hig})
\begin{displaymath}
\glob \circ \mathcal{L}=\Lambda \circ \glob, 
\end{displaymath}
where $\Lambda$ is the derivation operator on valuations. 

\section{The $\Sp(1)$-module of $\Sp(2)$-invariant curvature measures}
\label{sec_module_decomposition}

Starting with this section, we restrict our attention to a two-dimensional quaternionic right vector space $V$. The group $\Sp(2)$ acts from the left by usual matrix multiplication, while the group $\Sp(1)$ acts by scalar multiplication from the right. 

We are interested in $\spsp$-invariant elements (valuations, curvature measures, differential forms). It will be sometimes easier to describe the space of $\Sp(2)$-invariant elements as a representation of the group $\Sp(1)$. 

Since $\Sp(1) \cong \mathrm{SU}(2)$, irreducible representations of $\Sp(1)$ are indexed by their dimensions. Let $V_n$ be the unique (up to equivalence) complex irreducible representation of $\Sp(1)$ of dimension $n+1$. We recall the Clebsch-Gordan rule \cite[Exercice 11.11]{fulton_harris91}:
\begin{equation} \label{eq_clebsch_gordan}
 V_k \otimes V_l \cong V_{k+l} \oplus V_{k+l-2} \oplus \ldots \oplus V_{|k-l|}.
\end{equation}

\begin{Proposition} \label{prop_dimensions}
The $\Sp(1)$-modules $\Val_k^{\Sp(2)}$ and $\Curv_k^{\Sp(2)}$ decompose as follows
\begin{displaymath}
\begin{array}{c | c | c}
k & \Val_k^{\Sp(2)} & \Curv_k^{\Sp(2)}  \\ \hline 
0 &  V_0 & V_0\\
1 &  V_0 & 2V_0+V_4\\
2 &  2V_0+V_4& 4V_0+V_2+4V_4\\
3 &  3V_0+2V_4& 7V_0+3V_2+6V_4+V_6+V_8\\
4 &  5V_0+3V_4+V_8 & 7V_0+3V_2+6V_4+V_6+V_8\\
5 &  3V_0+2V_4& 4V_0+V_2+4V_4\\
6 &  2V_0+V_4& 2V_0+V_4\\
7 &  V_0 & V_0\\
8 &  V_0 & V_0
\end{array}
\end{displaymath}
In particular, the dimensions of the spaces of $k$-homogeneous invariant valuations and curvature measures are as follows
\begin{displaymath}
\begin{array}{c | c | c}
k & \dim \Val_k^{\spsp} & \dim \Curv_k^{\spsp}  \\ \hline 
0 & 1 & 1\\
1 & 1 & 2\\
2 & 2 & 4\\
3 & 3 & 7\\
4 & 5 & 7\\
5 & 3 & 4\\
6 & 2 & 2\\
7 & 1 & 1\\
8 & 1 & 1
\end{array}
\end{displaymath}
\end{Proposition}

\proof
The decomposition of $\Val_k^{\Sp(2)}$ as a sum of irreducible $\Sp(1)$-modules was found in \cite[Theorem 1.3]{bernig_qig}. Let us prove the statement for curvature measures, using the same notations as in \cite{bernig_qig}.

We fix a unit vector $v_0 \in V$ and let $\tilde V$ be the quaternionic orthogonal complement of $v_0 \h$. Then we may write 
\begin{displaymath}
 V = \R v_0 \bot U  \bot \tilde V,
\end{displaymath}
where $\dim U=3$.

Let $\mathrm{Stab}_{v_0}$ be the stabilizer in $\spsp$ of $v_0$. If we pick coordinates on $V$ such that $v_0=(1,0)$, then $\mathrm{Stab}_{v_0}$ is identified with $\Sp(1)\Sp(1)$ as follows.
Given an element $\pm(p,q)\in\Sp(1)\Sp(1)$, we let it act on $V$ by 
\[
 \pm(p,q)(\lambda+u,v)=\begin{pmatrix}q^{-1}&0\\0&p\end{pmatrix}\begin{pmatrix}\lambda +u\\v\end{pmatrix}q=(\lambda+q^{-1}uq,pvq)
\]
where $u\in U, v\in\tilde V$. In other words, $\Sp(1)\Sp(1)$ acts in the usual way on $\tilde V$. On $U$, the second factor acts by the adjoint representation (i.e. $U \cong V_2$), while the first factor acts trivially. Since $v_0$ is fixed, from here on we denote $\mathrm{Stab}_{v_0}$ by  $\Sp(1)\Sp(1)$.  

We set 
\begin{displaymath}
 R_{a,b}:=\left(\largewedge^a \tilde V^* \otimes \largewedge^b \tilde V^*\right)^{\Sp(1)} \otimes \C
\end{displaymath}
where the superscript $\Sp(1)$ denotes invariance with respect to the first $\Sp(1)$-factor.
The action of the second $\Sp(1)$-factor makes $R_{a,b}$ into an $\Sp(1)$-module. 

In \cite[Lemma 6.3]{bernig_qig}, it was computed that $R_{a,b} = 0$ if $a+b$ is odd and that 
\begin{align} 
 R_{0,0} & \cong R_{0,4} \cong R_{4,0} \cong R_{4,4} \cong V_0 \label{eq_r00}\\
 R_{0,2} & \cong R_{2,0} \cong R_{2,4} \cong R_{4,2} \cong V_2 \\
 R_{1,1} & \cong R_{1,3} \cong R_{3,1} \cong R_{3,3} \cong V_0+V_2\\
 R_{2,2} & \cong 2V_0+V_2+V_4. \label{eq_r22}
\end{align}

Let $\Omega_h^{k,l}(SV)^{\Sp(2),tr}$ denote the space of translation- and $\Sp(2)$-invariant complex-valued horizontal differential forms on the sphere bundle $SV$ of bidegree $(k,l)$. Then, as $\Sp(1)$-representations, 

\begin{equation} \label{eq_dec_forms}
 \Omega^{k,l}_h(SV)^{\Sp(2),tr} \cong \left(\largewedge^k(U^* \oplus \tilde V^*) \otimes \largewedge^l(U^* \oplus \tilde V^*)\right)^{\Sp(1)} \otimes \C.
\end{equation}
We clearly have 
\begin{equation} \label{eq_lambda_uv}
 \largewedge^j(U^* \oplus \tilde V^*) \cong \largewedge^j \tilde V^* \oplus (\largewedge^{j-1}\tilde V^* \otimes V_2) \oplus (\largewedge^{j-2} \tilde V^* \otimes V_2) \oplus \largewedge^{j-3}\tilde V^*.
\end{equation}
Expanding the right hand side of \eqref{eq_dec_forms}, using the values of $R_{a,b}$ from above and the Clebsch-Gordan rule \eqref{eq_clebsch_gordan}, we find the following table 

\begin{displaymath}
\begin{array}{c | c | c }
k & \Omega^{k,7-k}_h(SV)^{\Sp(2),tr} & \Omega^{k-1,6-k}_h(SV)^{\Sp( 2),tr}  \\ \hline 
0,7 & V_0 & 0 \\
1,6 & 2V_0+2V_2+V_4 & 2V_2 \\
2,5 & 7V_0+8V_2+7V_4+V_6 & 3V_0+7V_2+3V_4+V_6 \\
3,4 & 12V_0+18V_2+14V_4+4V_6+V_8 & 5V_0+15V_2+8V_4+3V_6 
\end{array}
\end{displaymath}
The decomposition of $\Curv_k^{\Sp(2)}$ now follows from  
\begin{displaymath} 
 \Curv_k^{\Sp(2)} \cong \Omega_h^{k,7-k}(SV)^{\Sp(2),tr}/\Omega_h^{k-1,6-k}(SV)^{\Sp(2),tr}.
\end{displaymath}

The dimensions given in the proposition follow from the fact that only $V_0$ contains a nontrivial $\Sp(1)$-invariant element.
\endproof

\section{Invariant forms on the sphere bundle}
\label{sec_inv_forms}

We use the notation from the previous section. 
We decompose 
\begin{align*}
V & \cong \R v_0 \oplus U \oplus \tilde V\\
T_{(p_0,v_0)}SV & \cong V \oplus U \oplus \tilde V \cong \R v_0 \oplus U \oplus \tilde V \oplus U \oplus \tilde V.
\end{align*}

We can split
\begin{align*}
\Omega_h^{k,l}&(SV)^{\Sp(2), tr}  \cong \left(\largewedge^k(U^* \oplus \tilde V^*) \otimes \largewedge^l(U^* \oplus \tilde V^*)\right)^{\Sp(1)} \otimes \C \\
& \cong \bigoplus_{k_1+k_2=k,l_1+l_2=l} \left(\largewedge^{k_1}U^* \otimes \largewedge^{k_2}\tilde V^* \otimes \largewedge^{l_1}U^* \otimes \largewedge^{l_2}\tilde V^*\right)^{\Sp(1)} \otimes \C\\
& \cong \bigoplus_{k_1+k_2=k,l_1+l_2=l} \largewedge^{k_1}U^* \otimes  \largewedge^{l_1}U^* \otimes \left( \largewedge^{k_2}\tilde V^* \otimes \largewedge^{l_2}\tilde V^*\right)^{\Sp(1)} \otimes \C\\
& =:\bigoplus_{k_1+k_2=k,l_1+l_2=l} \Omega^{k_1,k_2,l_1,l_2}.
\end{align*} 

Note that this decomposition is compatible with the wedge product and independent of the choice of $(p_0,v_0)$. 

To simplify the notation, we will not distinguish between a form $\omega \in \Omega_h(SV)^{\Sp(2),tr}$ and  its value $\omega|_{(p_0,v_0)} \in \largewedge^*(U^* \oplus \tilde V^* \oplus U^* \oplus \tilde V^*)^{\Sp(1)} \otimes \C$ at the point $(p_0,v_0)$.

Given $q \in S^2=\{\re=0\} \cap \Sp(1)$, consider the 1-forms
\begin{displaymath}
 \beta_q=\langle dz , \zeta q\rangle\in \Omega^{1,0,0,0},\qquad\gamma_q=\langle d\zeta,\zeta q\rangle\in \Omega^{0,0,1,0}
\end{displaymath}
where $(z,\zeta)$ denotes a generic element in $SV$.
 
\begin{Lemma} \label{lemma_basis_lambdau}
A basis for the algebra $\largewedge^*(U^* \oplus U^*)  \otimes \C$  is given by $\beta_{\mathbf i}, \beta_{\mathbf j}, \beta_{\mathbf k},$  ${\gamma_{\mathbf i}, \gamma_{\mathbf j}, \gamma_{\mathbf k}}$. The span of the $\beta$'s and the space of the $\gamma$'s are irreducible $\Sp(1)$-modules.
\end{Lemma}

\proof
We have 
\begin{displaymath}
\largewedge^k(U^* \oplus U^*)=\bigoplus_{a+b=k} \largewedge^a U^* \otimes \largewedge^b U^*,
\end{displaymath}
and both copies of $U^*$ are irreducible $\Sp(1)$-representations of dimension $3$. 
\endproof

Let us write 
\begin{displaymath}
\largewedge^{a,b}(U^* \oplus U^*):=\largewedge^a U^* \otimes \largewedge^b U^*.
\end{displaymath}

\begin{Corollary}
 The non-zero spaces $\largewedge^{a,b}(U^* \oplus U^*)  \otimes $, and their decomposition as $\Sp(1)$-modules, are given by the following table
\begin{displaymath}
 \begin{array}{c | c }
(a,b) & \largewedge^{a,b}(U^* \oplus U^*)  \otimes \C \\ \hline 
(0,0),(3,3),(0,3),(3,0) & V_0\\
(1,0),(0,1),(2,0),(0,2),(2,3),(3,2),(3,1),(1,3) & V_2\\
(1,1),(2,1),(1,2),(2,2) & V_0+V_2+V_4
\end{array}
\end{displaymath}
\end{Corollary}

\proof
This follows from Lemma \ref{lemma_basis_lambdau} and the Clebsch-Gordan rule \eqref{eq_clebsch_gordan}.
\endproof
The $\Sp(1)$-invariant elements corresponding to the $V_0$ components in the previous table are given by
\begin{align*}
 \phi_{1,1} & := \beta_{\mathbf i} \wedge \gamma_{\mathbf i}+\beta_{\mathbf j} \wedge \gamma_{\mathbf j}+\beta_{\mathbf k} \wedge \gamma_{\mathbf k} \\
\phi_{3,0} & :=\beta_{\mathbf i} \wedge \beta_{\mathbf j} \wedge \beta_{\mathbf k} \\ 
\phi_{2,1} & :=\beta_{\mathbf i} \wedge \beta_{\mathbf j} \wedge \gamma_{\mathbf k}+\beta_{\mathbf j} \wedge \beta_{\mathbf k} \wedge \gamma_{\mathbf i}+\beta_{\mathbf k} \wedge \beta_{\mathbf i} \wedge \gamma_{\mathbf j} \\
\phi_{1,2} & := \beta_{\mathbf i} \wedge \gamma_{\mathbf j} \wedge \gamma_{\mathbf k}+\beta_{\mathbf j} \wedge \gamma_{\mathbf k} \wedge \gamma_{\mathbf i}+\beta_{\mathbf k} \wedge \gamma_{\mathbf i} \wedge \gamma_{\mathbf j}\\
\phi_{0,3} & :=\gamma_{\mathbf i} \wedge \gamma_{\mathbf j} \wedge \gamma_{\mathbf k}.
\end{align*}
For $(a,b)=(3,3)$, the invariant element is  $\phi_{3,0} \wedge \phi_{0,3}$, and for $(a,b)=(2,2)$ it is $\phi_{1,1}^2$.

\bigskip
Given $q \in S^2$, let
\begin{align*}
 \tilde \theta_{0,q} & =\frac{1}{2}d\gamma_q=\frac{1}{2}\langle d \zeta q,d\zeta\rangle \in \Omega^{0,0,2,0} \oplus \Omega^{0,0,0,2} \\
\tilde \theta_{1,q} & =d\beta_q=\langle d \zeta q,dz\rangle \in \Omega^{1,0,1,0} \oplus \Omega^{0,1,0,1}\\
 \tilde \theta_{2,q} & =\frac{1}{2}\langle d z q,dz\rangle \in \Omega^{2,0,0,0} \oplus \Omega^{0,2,0,0},
\end{align*}
with the usual convention that the inner product of $1$-forms is skew symmetric. It may be checked in coordinates that no other weights appear.

Let $\theta_{0,q} \in \Omega^{0,0,0,2}, \theta_{1,q} \in \Omega^{0,1,0,1}, \theta_{2,q} \in \Omega^{0,2,0,0}$ be the corresponding projections. Moreover, let $\theta_s \in \Omega^{0,1,0,1}$ be the projection of $-d\alpha$, where $\alpha$ is the canonical $1$-form on $SV$. Then $\theta_s$ is a symplectic form on $\tilde V \oplus \tilde V$.

\begin{Lemma}\label{lemma_basis_lambdav}
A basis for the algebra of $\largewedge^*(\tilde V^* \oplus \tilde V^*)^{\Sp(1)}  \otimes \C$ is given by the union of the following sets:
\begin{displaymath}
\{ \theta_{0,\mathbf i}, \theta_{0,\mathbf j}, \theta_{0,\mathbf k}\}, \{\theta_{1,\mathbf i}, \theta_{1,\mathbf j}, \theta_{1,\mathbf k}\}, \{\theta_s\},  \{\theta_{2,\mathbf i}, \theta_{2,\mathbf j}, \theta_{2,\mathbf k}\}.
\end{displaymath}
The span of the elements in each of these sets is an irreducible $\Sp(1)$-module.
\end{Lemma}

\proof
Noting that $\mathrm{Sp}(1) \cong \mathrm{SU}(2)$, the statement follows from \cite[Lemma 3.3]{bernig_sun09}.
\endproof

\begin{Lemma} \label{lemma_relations} The following relations hold in $\largewedge^4(\tilde V^* \oplus \tilde V^*)^{\Sp(1)}$
 \begin{align}
  \theta_{0,\mathbf i}^2=\theta_{0,\mathbf j}^2 & =\theta_{0,\mathbf k}^2 \nonumber\\
  \theta_{0,\mathbf i} \wedge \theta_{0,\mathbf j} &=0 \nonumber \\
  \theta_{0,\mathbf i}\wedge\theta_{1,\mathbf i}=\theta_{0,\mathbf j}\wedge\theta_{1,\mathbf j} & =\theta_{0,\mathbf k}\wedge\theta_{1,\mathbf k} \nonumber\\
  \theta_{0,\mathbf i}\wedge\theta_{1,\mathbf j}=-\theta_{0,\mathbf j}\wedge\theta_{1,\mathbf i}&=-\theta_{0,\mathbf k}\wedge\theta_s \nonumber \\
  \theta_{0,\mathbf i}\wedge\theta_{2,\mathbf j}+\theta_{0,\mathbf j}\wedge\theta_{2,\mathbf i}-\theta_{1,\mathbf i}\wedge\theta_{1,\mathbf j}&=0 \nonumber \\
  \theta_{0,\mathbf i}\wedge\theta_{2,\mathbf j}-\theta_{0,\mathbf j}\wedge\theta_{2,\mathbf i}+\theta_{1,\mathbf k}\wedge\theta_s&=0 \nonumber \\
  \theta_{1,\mathbf i}^2-2\theta_{0,\mathbf j}\wedge\theta_{2,\mathbf j}-2\theta_{0,\mathbf k}\wedge\theta_{2,\mathbf k}-\theta_s^2&=0\label{theta1isquared}\\
  \theta_{1,\mathbf i}\wedge\theta_{2,\mathbf i}=\theta_{1,\mathbf j}\wedge\theta_{2,\mathbf j} & =\theta_{1,\mathbf k}\wedge\theta_{2,\mathbf k} \nonumber \\
  \theta_{1,\mathbf i}\wedge\theta_{2,\mathbf j}=-\theta_{1,\mathbf j}\wedge\theta_{2,\mathbf i}&=-\theta_{2,\mathbf k}\wedge\theta_{s} \nonumber \\
  \theta_{2,\mathbf i}^2 =\theta_{2,\mathbf j}^2 & =\theta_{2,\mathbf j}^2 \nonumber \\
  \theta_{2,\mathbf i}\wedge\theta_{2,\mathbf j}&=0.\nonumber
\end{align}
Also the equations obtained by cyclic permutation of $\mathbf i,\mathbf j,\mathbf k$ hold.
\end{Lemma}

\proof
This is a direct computation in coordinates  $z_1,\ldots,z_8,\zeta_1,\ldots,\zeta_8$ with $\sum \zeta_j^2=1$ of $S\h^2$. By invariance, it is enough to do it at a special point of $S\h^2$, e.g. $(0,1)$. At this point we have 
\begin{align*}
\theta_{0, \mathbf i} & = d\zeta_5 \wedge d\zeta_6-d\zeta_7 \wedge d\zeta_8\\
\theta_{1,\mathbf j} & = -d\zeta_7 \wedge dz_5-d\zeta_8 \wedge dz_6+d\zeta_5 \wedge dz_7+d\zeta_6 \wedge dz_8.
\end{align*}
Taking wedge products yields 
\begin{align*}
\theta_{0, \mathbf i} \wedge \theta_{1,\mathbf j} & = -d\zeta_5 \wedge d\zeta_7 \wedge d\zeta_8 \wedge dz_7-d\zeta_6 \wedge d\zeta_7 \wedge d\zeta_8 \wedge dz_8\\
& -d\zeta_5 \wedge d\zeta_6 \wedge d\zeta_7 \wedge dz_5-d\zeta_5 \wedge d\zeta_6 \wedge d\zeta_8 \wedge dz_6.
\end{align*}

Since 
\begin{align*}
\theta_{0, \mathbf k} & = d\zeta_5 \wedge d\zeta_8-d\zeta_6 \wedge d\zeta_7\\
\theta_s & = -d\zeta_5 \wedge dz_5 -d\zeta_6 \wedge dz_6 -d\zeta_7 \wedge dz_7 -d\zeta_8 \wedge dz_8,
\end{align*}
we obtain $\theta_{0, \mathbf i} \wedge \theta_{1,\mathbf j}=-\theta_{0,\mathbf k} \wedge \theta_s$.

The other equations are obtained similarly.
\endproof

We will denote
\begin{displaymath}
 \Theta_{m,n} :=\theta_{m,\mathbf i} \wedge \theta_{n,\mathbf i}+\theta_{m,\mathbf j} \wedge \theta_{n,\mathbf j}+\theta_{m,\mathbf k} \wedge \theta_{n,\mathbf k} \in \largewedge^{m+n,4-m-n}(\tilde V^* \oplus \tilde V^*)^{\Sp(1)}.
\end{displaymath} By the previous relations, $\Theta_{m,n}=3\theta_{m,\mathbf i} \wedge \theta_{n,\mathbf i}$ unless $m+n=2$. 

\begin{Lemma}\label{lemma_primitives}
 The following forms are primitive in the sense that their wedge product with $\theta_s$ vanishes:
\[
\Theta_{0,0},\Theta_{0,1},\theta_{1,\mathbf i}\wedge\theta_{1,\mathbf j},\theta_{1,\mathbf j}\wedge\theta_{1,\mathbf k},\theta_{1,\mathbf k}\wedge\theta_{1,\mathbf i},\theta_{1,\mathbf i}^2-\frac13\theta_s^2, \theta_{1,\mathbf j}^2-\frac13\theta_s^2, \theta_{1,\mathbf k}^2-\frac13\theta_s^2, \Theta_{1,2},\Theta_{2,2}.
\]
\end{Lemma}
\proof This is a direct computation  in coordinates.
\endproof

The content of the two previous lemmas is enough to find the Lefschetz decomposition of any element in $\largewedge^*(\tilde V^* \oplus \tilde V^*)^{\Sp(1)} \otimes \C$. For instance, the cyclic permutations of \eqref{theta1isquared}  gives the Lefschetz decomposition
\begin{equation}\label{eq_theta0i2i}
 \theta_{0,\mathbf i} \wedge \theta_{2,\mathbf i}=\frac14(\theta_{1,\mathbf j}^2+\theta_{1,\mathbf k}^2-\theta_{1,\mathbf i}^2-\theta_s^2)=\frac14\left(\theta_{1,\mathbf j}^2+\theta_{1,\mathbf k}^2-\theta_{1,\mathbf i}^2-\frac13\theta_s^2\right)-\frac16\theta_s^2.
\end{equation}

\begin{Lemma} \label{lemma_differentials}
The differentials of the basic forms are given by the equations
\begin{align*}
d\alpha & =-\beta_{\mathbf i} \wedge \gamma_{\mathbf i}-\beta_{\mathbf j} \wedge \gamma_{\mathbf j}-\beta_{\mathbf k} \wedge \gamma_{\mathbf k}-\theta_s\\
d\beta_{\mathbf i} & =\alpha \wedge \gamma_{\mathbf i}-\beta_{\mathbf j} \wedge \gamma_{\mathbf k}+\beta_{\mathbf k}  \wedge \gamma_{\mathbf j}+\theta_{1,\mathbf i}\\
d\gamma_{\mathbf i} & =-2\gamma_{\mathbf j} \wedge \gamma_{\mathbf k}+2\theta_{0,\mathbf i}\\
d\theta_{2,\mathbf i} & = \beta_{\mathbf i} \wedge \theta_s+\beta_{\mathbf k} \wedge \theta_{1,\mathbf j}-\beta_{\mathbf j} \wedge \theta_{1,\mathbf k}+\alpha \wedge \theta_{1, \mathbf i},
\end{align*}
and all equations which are obtained by cyclic permutations of $\mathbf i,\mathbf j,\mathbf k$ or by applying $d$ once more to the displayed equations. 
\end{Lemma}
Since we are interested in curvature measures, we will compute modulo $d\alpha$. By the first equation in Lemma \ref{lemma_differentials},  we may replace $\theta_s$ by $-\beta_{\mathbf i} \wedge \gamma_{\mathbf i}-\beta_{\mathbf j} \wedge \gamma_{\mathbf j}-\beta_{\mathbf k} \wedge \gamma_{\mathbf k}$, hence it is enough to consider forms without a $\theta_s$-factor.

\begin{Corollary}
 The non-zero spaces $\largewedge^{a,b}(\tilde V^* \oplus \tilde V^*)^{\Sp(1)} \otimes \C$ (modulo $\theta_s$), and their decomposition as $\Sp(1)$-modules, are given by the following table 
\begin{displaymath}
 \begin{array}{c | c }
(a,b) & R_{a,b}/(R_{a-1,b-1}\wedge\theta_s) \\ \hline 
(0,0),(4,0),(0,4),(4,4) & V_0\\
(1,1),(2,0),(0,2),(4,2),(2,4) & V_2\\
(1,3),(3,1)& V_0 \\
(2,2) & V_0+V_4 
\end{array}
\end{displaymath}
\end{Corollary}

\proof
This follows from \eqref{eq_r00}-\eqref{eq_r22} and the Lefschetz decomposition of $\largewedge^*(\tilde V^* \oplus \tilde V^*)^{\Sp(1)}  \otimes \C$ (recall that $\theta_s$ is a symplectic form). 
\endproof

It is easily checked that each $\Theta_{m,n}$ is $\Sp(1)\Sp(1)$-invariant, and thus belongs to the $V_0$ component of the corresponding entry in the previous table.

\bigskip
Using the fact that the tensor product $V_a \otimes V_b$ contains an invariant summand if and only if $a=b$, we can now determine which of the spaces $\Omega_{k_1,k_2,l_1,l_2}$ contain an $\Sp(1)$-invariant element. For instance, if $(k_1,l_1)\in\{(1,0),(0,1),(2,0),(0,2)\}$ and $(k_2,l_2)\in\{(2,0),(1,1),(0,2)\}$, then
\begin{displaymath}
 \Omega_{k_1,k_2,l_1,l_2} \cong V_2 \otimes V_2 \cong V_0+V_2+V_4
\end{displaymath}
contains an $\Sp(1)$-invariant element. These elements are listed below. Indeed, it is not difficult to check directly that these differential forms are $\spsp$-invariant.

\begin{align*}
\Phi_{\gamma,\theta_0} & := \gamma_{\mathbf i} \wedge \theta_{0,\mathbf i}+\gamma_{\mathbf j}  \wedge\theta_{0, \mathbf j}+\gamma_{\mathbf k} \wedge\theta_{0, \mathbf k} \in \Omega^{0,0,1,2}\\
\Phi_{\gamma,\theta_1} & := \gamma_{\mathbf i} \wedge \theta_{1, \mathbf i}+\gamma_{\mathbf j} \wedge \theta_{1,\mathbf j}+\gamma_{\mathbf k} \wedge\theta_{1,\mathbf k} \in \Omega^{0,1,1,1}\\
\Phi_{\gamma,\theta_2} & := \gamma_{\mathbf i} \wedge \theta_{2,\mathbf i}+\gamma_{\mathbf j} \wedge \theta_{2,\mathbf j}+\gamma_{\mathbf k} \wedge\theta_{2,\mathbf k} \in \Omega^{0,2,1,0}\\
\Phi_{\beta,\theta_0} & := \beta_{\mathbf i} \wedge \theta_{0,\mathbf i} +\beta_{\mathbf j} \wedge \theta_{0,\mathbf j} +\beta_{\mathbf k} \wedge \theta_{0,\mathbf k} \in \Omega^{1,0,0,2}\\
\Phi_{\beta,\theta_1} & := \beta_{\mathbf i} \wedge \theta_{1,\mathbf i} +\beta_{\mathbf j} \wedge \theta_{1,\mathbf j} +\beta_{\mathbf k} \wedge \theta_{1,\mathbf k} \in \Omega^{1,1,0,1}\\
\Phi_{\beta,\theta_2} & := \theta_{2,\mathbf i} \wedge \beta_{\mathbf i}+\theta_{2,\mathbf j} \wedge \beta_{\mathbf j}+\theta_{2,\mathbf k} \wedge \beta_{\mathbf k} \in \Omega^{1,2,0,0}\\
\Phi_{\beta, \beta, \theta_0} & := \beta_{\mathbf i} \wedge \beta_{\mathbf j} \wedge \theta_{0,\mathbf k}+\beta_{\mathbf j} \wedge \beta_{\mathbf k} \wedge \theta_{0,\mathbf i}+\beta_{\mathbf k} \wedge \beta_{\mathbf i} \wedge \theta_{0,\mathbf j} \in \Omega^{2,0,0,2}\\
\Phi_{\beta, \beta, \theta_1} & := \beta_{\mathbf i} \wedge \beta_{\mathbf j} \wedge \theta_{1,\mathbf k}+\beta_{\mathbf j} \wedge \beta_{\mathbf k} \wedge \theta_{1,\mathbf i}+\beta_{\mathbf k} \wedge \beta_{\mathbf i} \wedge \theta_{1,\mathbf j} \in \Omega^{2,1,0,1}\\
\Phi_{\beta, \beta, \theta_2} & := \beta_{\mathbf i} \wedge \beta_{\mathbf j} \wedge \theta_{2,\mathbf k}+\beta_{\mathbf j} \wedge \beta_{\mathbf k} \wedge \theta_{2,\mathbf i}+\beta_{\mathbf k} \wedge \beta_{\mathbf i} \wedge \theta_{2,\mathbf j} \in \Omega^{2,2,0,0}\\
\Phi_{\gamma,\gamma,\theta_0} & := \theta_{0,\mathbf i} \wedge \gamma_{\mathbf j} \wedge \gamma_{\mathbf k}+\theta_{0,\mathbf j} \wedge \gamma_{\mathbf k} \wedge \gamma_{\mathbf i}+\theta_{0,\mathbf k} \wedge \gamma_{\mathbf i} \wedge \gamma_{\mathbf j} \in \Omega^{0,0,2,2}\\
\Phi_{\gamma,\gamma,\theta_1} & := \theta_{1,\mathbf i} \wedge \gamma_{\mathbf j} \wedge \gamma_{\mathbf k}+\theta_{1,\mathbf j} \wedge \gamma_{\mathbf k} \wedge \gamma_{\mathbf i}+\theta_{1,\mathbf k} \wedge \gamma_{\mathbf i} \wedge \gamma_{\mathbf j} \in \Omega^{0,1,2,1}\\
\Phi_{\gamma,\gamma,\theta_2} & := \theta_{2,\mathbf i} \wedge \gamma_{\mathbf j} \wedge \gamma_{\mathbf k}+\theta_{2,\mathbf j} \wedge \gamma_{\mathbf k} \wedge \gamma_{\mathbf i}+\theta_{2,\mathbf k} \wedge \gamma_{\mathbf i} \wedge \gamma_{\mathbf j} \in \Omega^{0,2,2,0}.
\end{align*}

Together with the $\spsp$-invariant forms $\Phi_{a,b}$ and $\Theta_{m,n}$, the previous list will suffice to construct a basis of $\Curv^{\spsp}$. To this end, we define 
\begin{align*}
\Lambda_{0,1} & :=\phi_{0,3} \wedge \Theta_{0,0}\\
\Lambda_{1,1} & := \phi_{1,2} \wedge \Theta_{0,0}\\
\Lambda_{1,2} & := \phi_{0,3} \wedge \Theta_{0,1}\\
\Lambda_{2,1} & :=  \phi_{2,1} \wedge \Theta_{0,0}\\
\Lambda_{2,2} & :=  \phi_{1,2} \wedge \Theta_{0,1}\\
\Lambda_{2,3} & := \phi_{0,3} \wedge \left[\Theta_{0,2} +\frac12  (d\alpha)^2\right]\\
\Lambda_{2,4} & := \phi_{0,3} \wedge \Phi_{\beta, \beta, \theta_0}-\Phi_{\gamma,\theta_0} \wedge \left[\frac12 \phi_{1,1}+\frac16 d\alpha\right] \wedge d\alpha\\
\Lambda_{3,1} & := \phi_{3,0} \wedge \Theta_{0,0}\\
 \Lambda_{3,2} & := \phi_{2,1} \wedge \Theta_{0,1}\\
 \Lambda_{3,3} & := \phi_{1,2} \wedge \left[4\Theta_{0,2}+2 (d\alpha)^2\right] \\
 \Lambda_{3,4} & := \Phi_{\beta,\theta_0} \wedge \Phi_{\gamma,\gamma,\theta_2} + \Phi_{\beta,\theta_2} \wedge \Phi_{\gamma,\gamma,\theta_0}+\frac13 \phi_{1,2} \wedge d\alpha \wedge d\alpha\\
 \Lambda_{3,5} & := \phi_{0,3} \wedge \Theta_{1,2} \\
 \Lambda_{3,6} & := \phi_{3,0} \wedge \Phi_{\gamma,\gamma,\theta_0}- \Phi_{\beta,\theta_0} \wedge d\alpha \wedge \left[\frac12 \phi_{1,1}+\frac16 d\alpha\right]\\
 \Lambda_{3,7} & := \phi_{0,3} \wedge \Phi_{\beta,\beta,\theta_1} -\Phi_{\gamma,\theta_1} \wedge d\alpha \wedge \left[\frac12 \phi_{1,1} +\frac16 d\alpha\right]\\   
\Lambda_{4,1} & := \phi_{3,0} \wedge  \Theta_{0,1}\\
  \Lambda_{4,2} & := \phi_{2,1} \wedge \left[\Theta_{0,2}+\frac12 (d\alpha)^2\right]\\
  \Lambda_{4,3} & := \Phi_{\beta,\beta,\theta_0} \wedge \Phi_{\gamma,\theta_2} + \Phi_{\beta,\beta,\theta_2} \wedge \Phi_{\gamma,\theta_0} +\frac13 \phi_{2,1} \wedge d\alpha \wedge d\alpha\\
  \Lambda_{4,4} & := \phi_{1,2} \wedge \Theta_{1,2}\\
  \Lambda_{4,5} & := \phi_{0,3} \wedge \Theta_{2,2}\\
  \Lambda_{4,6} & := \phi_{3,0} \wedge \Phi_{\gamma,\gamma,\theta_1} -\Phi_{\beta,\theta_1} \wedge d\alpha \wedge\left[\frac12 \phi_{1,1}+\frac16 d\alpha\right]\\
  \Lambda_{4,7} & := \phi_{0,3} \wedge \Phi_{\beta,\beta,\theta_2} - \Phi_{\gamma,\theta_2} \wedge d\alpha \wedge \left[\frac12 \phi_{1,1} +\frac16 d\alpha\right]\\
\Lambda_{5,1} & := \phi_{3,0} \wedge \left[\Theta_{0,2}+\frac12 (d\alpha)^2\right]\\
   \Lambda_{5,2} & := \phi_{2,1} \wedge \Theta_{1,2}\\
   \Lambda_{5,3} & := \phi_{1,2} \wedge \Theta_{2,2}\\
   \Lambda_{5,4} & := \phi_{3,0} \wedge \Phi_{\gamma,\gamma,\theta_2} -\Phi_{\beta,\theta_2} \wedge d\alpha \wedge \left[\frac12 \phi_{1,1}+\frac16  d\alpha\right]\\
   \Lambda_{6,1} & := \phi_{3,0} \wedge \Theta_{1,2}\\
    \Lambda_{6,2} & := \phi_{2,1} \wedge \Theta_{2,2}\\
    \Lambda_{7,1} & := \phi_{3,0} \wedge \Theta_{2,2}.
\end{align*}
The terms containing the symplectic form $d\alpha$ are chosen in such a way that each of these forms is primitive, which means in this case in the kernel of the multiplication by $d\alpha$. Although this is irrelevant for the induced curvature measures, it simplifies some of the computations below. 

As an example, let us show that $\Lambda_{2,3}$ is primitive. Let us denote by curly brackets the sum of all terms obtained by cyclic permutations of $\mathbf{i},\mathbf j,\mathbf k$, so that e.g. $\phi_{1,1}=\{\beta_{\mathbf i} \wedge \gamma_{\mathbf i}\}$. 

By \eqref{eq_theta0i2i} we have 
\begin{align*}
\Lambda_{2,3} \wedge d\alpha & = -\phi_{0,3} \wedge \left(\{\theta_{0,\mathbf i} \wedge \theta_{2,\mathbf i}\}+\frac12 \theta_s^2\right) \wedge \theta_s\\
& = -\phi_{0,3} \wedge \left\{\theta_{0,\mathbf i} \wedge \theta_{2,\mathbf i}+\frac16 \theta_s^2\right\} \wedge \theta_s\\
 & = 0.
\end{align*}

\begin{Proposition}
The space $\Curv_k^{\spsp}$ is spanned by the curvature measures corresponding to the $\Lambda_{k,i}$.
\end{Proposition}

\proof
For each given $k$, the displayed forms $\Lambda_{k,i}$ are linearly independent and primitive. Since their number equals the dimension of the space of invariant curvature measures of degree $k$, the statement follows. 
\endproof

In the following, we will not distinguish between the form $\Lambda_{k,i}$ and the curvature measure induced by it.
\section{The Rumin differential}
\label{sec_rumin}

Let $(M,Q)$ be a contact manifold of dimension $2m+1$. The Rumin differential was introduced in \cite{rumin94}. It is an operator $D:\Omega^m(M) \to \Omega^{m+1}(M)$ defined as follows. Let $\alpha$ be a locally defining $1$-form, i.e. $Q=\ker \alpha$. 

Let $\omega \in \Omega^m(M)$. Then there exists a unique vertical form $\alpha \wedge \xi$ such that $D\omega:=d(\omega+\alpha \wedge \xi)$ is vertical. 

To find $\xi$, we restrict $d(\omega+\alpha \wedge \xi)$ to the contact plane and find
\begin{displaymath}
\left(d\omega+d\alpha \wedge \xi\right)|_Q=0.
\end{displaymath}
Since multiplication by the symplectic form $d\alpha$ is an isomorphism between $(m-1)$-forms and $(m+1)$-forms, we may take 
\begin{displaymath}
\xi|_Q:=- (d\alpha)^{-1} (d\omega).
\end{displaymath}

In our situation, $M=S\h^2$ is the sphere bundle with its global contact form $\alpha$, and $m=7$. Let us describe an algorithm to compute the Rumin differential of the forms $\Lambda_{k,i}$. Since the image of $D$ consists of vertical forms, the knowledge of $D$ is equivalent to the knowledge of the operator $\mathcal{L}$ from Definition \ref{def_derivation_operator} which is given by $i_TD$, where $T$ is the Reeb vector field (i.e. $i_T\alpha = 1, i_T d\alpha=0$). 

Let $\omega \in \Omega^7(S\h^2)$ be translation- and $\spsp$-invariant. It is easy to compute $d\omega$ by Lemma \ref{lemma_differentials}. The hard part is to divide by $d\alpha$. Since this is a linear operator, we may do the computation at some fixed point $(p_0,v_0) \in S\h^2$. The contact plane at this point is given by the $14$-dimensional $Q_{(p_0,v_0)}=v_0^\perp \oplus v_0^\perp$. 

As before, we split the $7$-dimensional space $v_0^\perp$ into a $3$-dimensional space $U$ inside $v_0 \cdot \h$ and the $4$-dimensional quaternionic orthogonal complement $\tilde V$. Consequently, the space of alternating $k$-forms splits as 
\begin{equation} \label{eq_quaternionic_splitting}
 \largewedge^k(U^* \oplus \tilde V^* \oplus U^* \oplus \tilde V^*) \cong \bigoplus_{a+b=k} \largewedge^a(U^* \oplus U^*) \otimes \largewedge^b(\tilde V^* \oplus \tilde V^*).
\end{equation}
We call the degree in the $\tilde V$-part the {\it quaternionic degree}. Note that the quaternionic degree of invariant forms is always even. 

The algorithm takes a form $\eta:=d\omega|_{Q_{(p_0,v_0)}}$ of total degree $8$ and quaternionic degree $b$ and returns $(d\alpha)^{-1}\eta$. At each step, it reduces the quaternionic degree $b$.  

The symplectic form $\Omega:=-d\alpha|_{Q_{(p_0,v_0)}}$ splits as $\Omega=\Omega_1+\Omega_2$ with $\Omega_1:=\phi_{1,1}|_{Q_{(p_0,v_0)}} \in \largewedge^2(U^* \oplus U^*)$ a symplectic form on $U \oplus U$ and $\Omega_2:=\theta_s|_{Q_{(p_0,v_0)}} \in \largewedge^2(\tilde V^* \oplus \tilde V^*)$ a symplectic form on $\tilde V\oplus \tilde V$.  

Let $\eta=\eta_1 \wedge \eta_2$ be the decomposition of $\eta$ according to the splitting \eqref{eq_quaternionic_splitting}. Since $\eta_2 \in \largewedge^b(\tilde V^* \oplus \tilde V^*)$, we may write $\eta_2=\pi_2 + \Omega_2 \wedge \rho_2$, where $\pi_2,\rho_2 \in \largewedge^*(\tilde V \oplus \tilde V)$ and where $\pi_2$ is primitive, i.e. $\Omega_2^{5-b} \wedge \pi_2=0$ if $b\leq 4$ and $\pi_2=0$ if $b\geq 6$. Given $\eta_2$, it is easy to obtain this decomposition explicitly using  Lemmas \ref{lemma_relations} and \ref{lemma_primitives}.

We may write 
\begin{align*}
\eta_1 \wedge \eta_2 & =\eta_1 \wedge \pi_2+\eta_1 \wedge \Omega_2 \wedge \rho_2\\
& =\eta_1 \wedge \pi_2+\eta_1 \wedge \Omega \wedge \rho_2-\eta_1 \wedge \Omega_1 \wedge \rho_2.
\end{align*}
The second term is obviously a multiple of the symplectic form $\Omega$. The last term has smaller quaternionic degree than $\eta$ and can be treated inductively. Indeed, this term vanishes if $b\leq 4$.

It remains to consider the case $\eta=\eta_1 \wedge \pi_2$ with $\pi_2$ primitive. Note that the degree of a primitive form is at most $4$. Since $U \oplus U$ is $6$-dimensional,  we have $a \leq 6$, hence $b=8-a \geq 2$, so that $\pi_2$ can be of degree $2$ or $4$. 

Suppose first that $\pi_2$ is of degree $4$. Then $\Omega_2 \wedge \pi_2=0$ and $\eta_1$ is of degree $4$.  We may therefore write 
\begin{displaymath}
\eta_1=\Omega_1 \wedge \rho_1,
\end{displaymath} 
and it is easy to find $\rho_1$ explicitly. This yields 
\begin{displaymath}
\eta_1 \wedge \pi_2=\Omega_1 \wedge \rho_1 \wedge \pi_2=(\Omega_1+\Omega_2) \wedge \rho_1 \wedge \pi_2=\Omega \wedge \rho_1 \wedge \pi_2.
\end{displaymath}

Let now $\pi_2$ be of degree $2$. Then $\Omega_2^3 \wedge \pi_2=0$. Since $\eta_1$ is of degree $6$, it is a multiple of $\Omega_1^3$. Using 
\begin{align*}
 \Omega_1^3 \wedge \pi_2 & =(\Omega-\Omega_2)^3 \wedge \pi_2\\
 & =(\Omega^3 -3\Omega^2 \wedge \Omega_2+3\Omega \wedge \Omega_2^2-\Omega_2^3) \wedge \pi_2\\
 & = \Omega \wedge (\Omega^2-3\Omega \wedge \Omega_2+3\Omega_2^2) \wedge \pi_2,
\end{align*}
we may divide by $\Omega$.

The result of the algorithm is as follows.

\begin{Proposition} \label{prop_derivation_operator}
The derivation operator is given by 
\begin{align*}
 \mathcal{L}\Lambda_{1,1} & = 3\Lambda_{0,1} \\
 \mathcal{L}\Lambda_{1,2} & =2\Lambda_{0,1}\\
 \mathcal{L}\Lambda_{2,1} & = 2\Lambda_{1,1}\\
 \mathcal{L}\Lambda_{2,2} & = -4\Lambda_{1,1}+12 \Lambda_{1,2} \\
 \mathcal{L}\Lambda_{2,3} & =  2\Lambda_{1,1}-2 \Lambda_{1,2} \\
 \mathcal{L}\Lambda_{2,4} & =  -\frac43 \Lambda_{1,1}+2 \Lambda_{1,2}\\
 \mathcal{L}\Lambda_{3,1} & =\Lambda_{2,1}\\
 \mathcal{L}\Lambda_{3,2} & = -10\Lambda_{2,1}+8\Lambda_{2,2}-36\Lambda_{2,4}\\
 \mathcal{L}\Lambda_{3,3} & = 16 \Lambda_{2,1}-12\Lambda_{2,2}+ 60\Lambda_{2,3}+ 144\Lambda_{2,4} \\
 \mathcal{L}\Lambda_{3,4} & = -\frac{32}{3} \Lambda_{2,1}+\frac{34}{3}\Lambda_{2,2}-30\Lambda_{2,3}-96\Lambda_{2,4}\\
 \mathcal{L}\Lambda_{3,5} & = 6 \Lambda_{2,2}-30 \Lambda_{2,3}-72 \Lambda_{2,4}\\
 \mathcal{L}\Lambda_{3,6} & = -\frac{4}{3} \Lambda_{2,1}+\frac23 \Lambda_{2,2}-3 \Lambda_{2,4} \\
 \mathcal{L}\Lambda_{3,7} & = -\frac{8}{3} \Lambda_{2,1}+\frac43 \Lambda_{2,2}-6\Lambda_{2,4}\\
\mathcal{L}\Lambda_{4,1} & = -16\Lambda_{3,1}+4\Lambda_{3,2}-36\Lambda_{3,6}\\
 \mathcal{L}\Lambda_{4,2} & = 6\Lambda_{3,1}-4\Lambda_{3,2}+2\Lambda_{3,3}+3\Lambda_{3,4}+36\Lambda_{3,6}\\
 \mathcal{L}\Lambda_{4,3} & = 44\Lambda_{3,1}-16\Lambda_{3,2}+4\Lambda_{3,3}+6\Lambda_{3,4}+144\Lambda_{3,6}\\
 \mathcal{L}\Lambda_{4,4} & = 12\Lambda_{3,2}-3\Lambda_{3,3}-18\Lambda_{3,4}+12\Lambda_{3,5}-72\Lambda_{3,6}+36\Lambda_{3,7}\\
 \mathcal{L}\Lambda_{4,5} & = 18\Lambda_{3,4}-16\Lambda_{3,5}-72\Lambda_{3,7}\\
 \mathcal{L}\Lambda_{4,6} & = -8\Lambda_{3,1}+\frac43 \Lambda_{3,2}-44\Lambda_{3,6}+16\Lambda_{3,7}\\
 \mathcal{L}\Lambda_{4,7} & = -\frac43 \Lambda_{3,2}+\frac12 \Lambda_{3,3}+3\Lambda_{3,4}-2\Lambda_{3,5}+38\Lambda_{3,6}-22\Lambda_{3,7}\\
 \mathcal{L}\Lambda_{5,1} & = -5\Lambda_{4,1}+7\Lambda_{4,2}-3\Lambda_{4,3}\\
 \mathcal{L}\Lambda_{5,2} & = 18\Lambda_{4,1}-30\Lambda_{4,2}+18\Lambda_{4,3}+8\Lambda_{4,4}+36\Lambda_{4,6}+36\Lambda_{4,7}\\
 \mathcal{L}\Lambda_{5,3} & = 36\Lambda_{4,2}-18\Lambda_{4,3}-10\Lambda_{4,4}+3\Lambda_{4,5}-72\Lambda_{4,6}-72\Lambda_{4,7} \\
 \mathcal{L}\Lambda_{5,4} & = -4\Lambda_{4,1}+6\Lambda_{4,2}-3\Lambda_{4,3}-\frac23\Lambda_{4,4}-3\Lambda_{4,6}-3\Lambda_{4,7}\\
 \mathcal{L}\Lambda_{6,1} & = -12\Lambda_{5,1}+4 \Lambda_{5,2}+36\Lambda_{5,4}\\
 \mathcal{L}\Lambda_{6,2} & = 36\Lambda_{5,1}-4 \Lambda_{5,2}+2\Lambda_{5,3}-72\Lambda_{5,4}\\
 \mathcal{L}\Lambda_{7,1} & =2\Lambda_{6,1}+\Lambda_{6,2}.
\end{align*}
\end{Proposition}

\proof
We compute only $\mathcal L \Lambda_{2,2}$. This is the first case in the list where $d \Lambda_{k,q}$ is not vertical. Let us begin with
\begin{equation}\label{dLambda22}
 d\Lambda_{2,2}=d\phi_{1,2} \wedge \Theta_{0,1}-\phi_{1,2}\wedge d\Theta_{0,1}.
\end{equation}

Using Lemma \ref{lemma_differentials} we find
\begin{align*}
d\phi_{1,2}=\{\gamma_\mathbf i\wedge\gamma_\mathbf j\wedge\theta_{1,\mathbf k}\}+2\{\beta_\mathbf i\wedge(\gamma_\mathbf j\wedge\theta_{0,\mathbf k}-\gamma_\mathbf k\wedge\theta_{0,\mathbf j})\}+3\alpha\wedge\phi_{0,3}
\end{align*}
By the second and third relations in Lemma \ref{lemma_relations}, we get
\begin{displaymath}
d\phi_{1,2} \wedge\Theta_{0,1}=\{\gamma_\mathbf i\wedge\gamma_\mathbf j\wedge\theta_{1,\mathbf k}\} \wedge\Theta_{0,1}+3\alpha\wedge\phi_{0,3}\wedge\Theta_{0,1}.
\end{displaymath}
 
Let 
\begin{displaymath}
\eta_1:=\gamma_\mathbf i\wedge\gamma_\mathbf j,\qquad \eta_2:=\theta_{1,\mathbf k}\wedge\Theta_{0,1}=3\theta_{0\mathbf k}\wedge\theta_s^2,
\end{displaymath}
where we used the third and fourth relations in Lemma \ref{lemma_relations}. 
Using $\theta_s=-d\alpha-\phi_{1,1}$, we get
\begin{displaymath}
 \eta_1\wedge\eta_2=3 \eta_1\wedge\theta_{0\mathbf k}\wedge (d\alpha^2+2d\alpha\wedge\phi_{1,1}).
\end{displaymath}
Cyclic permutation yields
\begin{displaymath}
d\phi_{1,2} \wedge\Theta_{0,1}=3 d\alpha\wedge(d\alpha+2\phi_{1,1})\wedge \{\gamma_\mathbf i\wedge\gamma_\mathbf j\wedge\theta_{0,\mathbf k}\}+3\alpha\wedge\phi_{0,3}\wedge\Theta_{0,1}.
\end{displaymath}

Using Lemmas \ref{lemma_differentials} and \ref{lemma_relations}, the second term in \eqref{dLambda22} is found to be
\begin{align*}
 \phi_{1,2}\wedge d\Theta_{0,1}&=-2\alpha\wedge \phi_{1,2}\wedge\Theta_{0,0}+3\theta_s\wedge\phi_{1,2}\wedge \{\gamma_{\mathbf i}\wedge \theta_{0,\mathbf i}\}\\
 &=-2\alpha\wedge \phi_{1,2}\wedge\Theta_{0,0}-3 d\alpha\wedge\phi_{1,2}\wedge \{\gamma_{\mathbf i}\wedge \theta_{0,\mathbf i}\},
\end{align*}
since $\phi_{1,1}\wedge\phi_{1,2}$ vanishes. 

We obtain that
\begin{equation}\label{eq_LLambda22}
 D\Lambda_{2,2}=d(\Lambda_{2,2}+\alpha\wedge\xi)=\alpha\wedge(2\phi_{1,2}\wedge\Theta_{0,0}+3\phi_{0,3}\wedge\Theta_{0,1}-d\xi)
\end{equation}
where
\begin{align*}
 \xi&=-3(d\alpha+2\phi_{1,1}) \wedge \{\gamma_\mathbf i\wedge\gamma_\mathbf j\wedge\theta_{0,\mathbf k}\}-3 \phi_{1,2}\wedge \{\gamma_\mathbf i\wedge \theta_{0,\mathbf i}\}\\
 &=-3d\alpha\wedge \{\gamma_\mathbf i\wedge\gamma_\mathbf j\wedge\theta_{0,\mathbf k}\}+9\phi_{0,3}\wedge \{\beta_\mathbf i\wedge \theta_{0,\mathbf i}\}.
\end{align*}
Finally, using
\begin{displaymath}
 d\phi_{0,3}=2 \{\gamma_\mathbf i\wedge\gamma_\mathbf j\wedge\theta_{0,\mathbf k}\},
\end{displaymath}
we find
\begin{align*}
 d\xi &= 18 \{\gamma_\mathbf i\wedge\gamma_\mathbf j\wedge\theta_{0,\mathbf k}\} \wedge \{\beta_\mathbf i\wedge \theta_{0,\mathbf i}\}-9\phi_{0,3}\wedge\Theta_{0,1} \\
 &=18\{\gamma_\mathbf i\wedge\gamma_\mathbf j\wedge \beta_k\wedge\theta_{0,\mathbf k}^2\}-9\phi_{0,3}\wedge\Theta_{0,1}\\
 &=6\phi_{1,2}\wedge\Theta_{0,0}-9\phi_{0,3}\wedge\Theta_{0,1}.
\end{align*}
Plugging this in \eqref{eq_LLambda22}, and contracting with $T$ yields the stated identity.
\endproof

\section{Computation of the Klain functions}
\label{sec_klain_functions}

In \cite{bernig_solanes} we determined those modules appearing in the decomposition of $\Val_k(\mathbb H^2)$ from Theorem \ref{thm_alesker_bernig_schuster} which contain non-trivial $\spsp$-invariant elements. These modules are $\Gamma_{0,0,0,0},\Gamma_{2,2,0,0},\Gamma_{4,2,2,0},\Gamma_{2,2,2,2},\Gamma_{6,2,2,-2}$, and each of them contains a $1$-dimensional space of $\spsp$-invariant elements. In fact, in \cite{bernig_solanes}, we did not distinguish between $ \Gamma_{2,2,2,2}$ and $\Gamma_{2,2,2,-2}$ as well as between $\Gamma_{6,2,2,-2}$ and $\Gamma_{6,2,2,2}$. 
Using characters, it can be checked that the displayed signs are the correct ones. More precisely, Weyl's character formula enables us to compute the character of an irreducible $\mathrm{SO}(8)$ or $\spsp$-module. We can then  decompose $\Gamma_{6,2,2,2}$ and $\Gamma_{6,2,2,-2}$ as sums of irreducible $\spsp$-modules. The trivial representation only appears in the second case, that of $\Gamma_{6,2,2-2}$.

The multiplicities of  these modules in $\Curv_k$ can be found using recent results by Saienko \cite{saienko_thesis}. By Section \ref{sec_symmetry} we may restrict to $0 \leq k \leq 3$ and obtain the following multiplicities.

\begin{displaymath}
\begin{array}{c | c c c c c}
& \Gamma_{0,0,0,0} & \Gamma_{2,2,0,0} & \Gamma_{4,2,2,0} &\Gamma_{2,2,2,2} & \Gamma_{6,2,2,-2}\\ \hline
\Curv_0 & 1 & 0 & 0 & 0 & 0\\
\Curv_1 & 1 & 1 & 0 & 0 & 0\\
\Curv_2 & 1 & 2 & 1 & 0 & 0\\
\Curv_3 & 1 & 2 & 2 & 1 & 1
\end{array}
\end{displaymath}
 
Comparing with the dimensions of $\Curv_k^{\spsp}$ from Proposition \ref{prop_dimensions}, we see that these are all modules containing invariant curvature measures.

We claim that these invariant curvature measures are given by the following list. 
\begin{align*}
v_1^0 & :=\Lambda_{0,1} \\
v_1^1 & :=\Lambda_{1,1}+2\Lambda_{1,2} \\
v_2^1 & :=2\Lambda_{1,1}-3\Lambda_{1,2} \\
v_1^2 & := -\Lambda_{2,1}-2\Lambda_{2,2}-6\Lambda_{2,3} \\
v_2^2 & := -\frac{1}{12} \Lambda_{2,2}+\frac{1}{2} \Lambda_{2,3}-\Lambda_{2,4} \\
v_3^2 & := \frac{17}{63} \Lambda_{2,1}-\frac{1}{63} \Lambda_{2,2}-\frac{5}{7} \Lambda_{2,3}-\Lambda_{2,4}  \\
v_4^2 & := -\frac19 \Lambda_{2,1}+\frac19 \Lambda_{2,2}-\frac13 \Lambda_{2,3}-\Lambda_{2,4}\\
v_1^3 & :=\frac12 \Lambda_{3,1}+\Lambda_{3,2}+\frac34\Lambda_{3,3}+\Lambda_{3,5} \\
v_2^3 & :=\Lambda_{3,2} -\Lambda_{3,5}+2\Lambda_{3,6}+2\Lambda_{3,7} \\
v_3^3 & :=-\frac37 \Lambda_{3,1}-\frac{4}{21} \Lambda_{3,2}+\frac{3}{28}\Lambda_{3,3}+\frac17\Lambda_{3,5}+\Lambda_{3,6}+\Lambda_{3,7}\\
v_4^3 & :=\frac{1}{3} \Lambda_{3,2}-\frac{1}{2}\Lambda_{3,3}+\Lambda_{3,4}+\Lambda_{3,5}+2\Lambda_{3,6}+2\Lambda_{3,7} \\ 
v_5^3 & :=\frac{23}{45} \Lambda_{3,1}-\frac{4}{45} \Lambda_{3,2}+\frac{1}{180}\Lambda_{3,3}-\frac{4}{15}\Lambda_{3,4}+\frac{11}{45}\Lambda_{3,5}+\Lambda_{3,6}+\Lambda_{3,7} \\
v_6^3 & :=\frac{2}{5} \Lambda_{3,1}-\frac{1}{5} \Lambda_{3,2}+\frac{1}{20}\Lambda_{3,3}+\frac{3}{10}\Lambda_{3,4}-\frac{1}{5}\Lambda_{3,5}-2\Lambda_{3,6}+\Lambda_{3,7} \\
v_7^3 & :=-\frac{2}{9} \Lambda_{3,1}+\frac{1}{9} \Lambda_{3,2}-\frac{1}{36}\Lambda_{3,3}-\frac{1}{6}\Lambda_{3,4}+\frac{1}{9}\Lambda_{3,5}-2\Lambda_{3,6}+\Lambda_{3,7} \\
v_1^4 & :=-\Lambda_{4,1}-3\Lambda_{4,2}-\Lambda_{4,4}-\frac12 \Lambda_{4,5} \\
v_2^4 & :=\Lambda_{4,1} -\frac13 \Lambda_{4,4}-2\Lambda_{4,6}-2\Lambda_{4,7} \\
v_3^4 & :=-\frac17 \Lambda_{4,1}-\frac37 \Lambda_{4,2} +\frac{4}{21} \Lambda_{4,4}+\frac37 \Lambda_{4,5}-\Lambda_{4,6}-\Lambda_{4,7} \\
v_4^4 & :=-\Lambda_{4,1}+2 \Lambda_{4,2}-\Lambda_{4,3} -\frac13 \Lambda_{4,4}-2 \Lambda_{4,6}-2\Lambda_{4,7} \\
\end{align*}
\begin{align*}
v_5^4 & :=-\frac{11}{45} \Lambda_{4,1}-\frac{1}{45} \Lambda_{4,2}+\frac{4}{15} \Lambda_{4,3} +\frac{4}{45} \Lambda_{4,4}-\frac{23}{45}  \Lambda_{4,5}-\Lambda_{4,6}-\Lambda_{4,7} \\
v_6^4 & :=\frac15 \Lambda_{4,1}-\frac15 \Lambda_{4,2}-\frac{3}{10} \Lambda_{4,3} +\frac15 \Lambda_{4,4}-\frac25 \Lambda_{4,5}-\Lambda_{4,6}+2\Lambda_{4,7} \\
v_7^4 & :=-\frac19 \Lambda_{4,1}+\frac19 \Lambda_{4,2}+\frac16 \Lambda_{4,3} -\frac19 \Lambda_{4,4}+\frac29 \Lambda_{4,5}-\Lambda_{4,6}+2\Lambda_{4,7} \\
v_1^5 & := 6\Lambda_{5,1}+2\Lambda_{5,2}+\Lambda_{5,3} \\
v_2^5 & := -\frac12 \Lambda_{5,1}+\frac{1}{12} \Lambda_{5,2}+\Lambda_{5,4} \\
v_3^5 & := \frac57 \Lambda_{5,1}+\frac{1}{63} \Lambda_{5,2}-\frac{17}{63} \Lambda_{5,3}+\Lambda_{5,4}  \\
v_4^5 & := \frac13 \Lambda_{5,1}-\frac19 \Lambda_{5,2}+\frac19 \Lambda_{5,3}+\Lambda_{5,4} \\
v_1^6 & := 2\Lambda_{6,1}+\Lambda_{6,2} \\
v_2^6 & := -\frac32 \Lambda_{6,1}+\Lambda_{6,2} \\
v_1^7 &:=\Lambda_{7,1} 
\end{align*}

\begin{Lemma} \label{lemma_derivation_operator}
\begin{align*}
 \mathcal{L}(v^0_1) &= 0\\
 \mathcal{L}(v^1_1) &= 7v^0_1,\  \mathcal{L}(v^1_2) = 0\\
 \mathcal{L}(v^2_1) &= -6v^1_1,\  \mathcal{L}(v^2_2) = \frac43v^1_2,\  \mathcal{L}(v^2_3) = \frac{16}{63}v^1_2,\  \mathcal{L}(v^2_4) = 0\\
 \mathcal{L}(v^3_1) &= -\frac52v^2_1,\  \mathcal{L}(v^3_2) = -42v^2_3,\  \mathcal{L}(v^3_3) = -3v^2_3,\ \mathcal{L}(v^3_4) = 270v^2_4,\  \mathcal{L}(v^3_5) = -3v^2_4,\ \\ \mathcal{L}(v^3_6) &=0,\   \mathcal{L}(v^3_7) = 0\\
 \mathcal{L}(v^4_1) &= -4v^3_1,\  \mathcal{L}(v^4_2) = 0,\  \mathcal{L}(v^4_3) = -18v^3_3,\  \mathcal{L}(v^4_4) = 0,\  \mathcal{L}(v^4_5) = 46v^3_5,\ \\ \mathcal{L}(v^4_6) &= -24v^3_6,\  \mathcal{L}(v^4_7) = -80v^3_7\\
 \mathcal{L}(v^5_1) &= -6v^4_1,\  \mathcal{L}(v^5_2) = 0,\  \mathcal{L}(v^5_3) = -\frac{68}{9}v^4_2-\frac{17 }{9}v^4_3,\  \mathcal{L}(v^5_4) = \frac{180}{23}v^4_4-\frac{15}{23}v^4_5\\
 \mathcal{L}(v^6_1) &= 2v^5_1,\  \mathcal{L}(v^6_2) = -\frac{2016}{17}v^5_2-\frac{126}{17}v^5_3\\
 \mathcal{L}(v^7_1) &= v^6_1.
\end{align*}
\end{Lemma}

\proof
This is straightforward using Proposition \ref{prop_derivation_operator}.
\endproof

Define an operator 
\begin{displaymath}
 P_k:\Curv_k(\R^8) \to \Curv_k(\R^8), P:=\begin{cases} \mathcal{L}^{7-2k} \circ I^* & 0 \leq k \leq 3\\ I^* \circ \mathcal{L}^{2k-7} & 4 \leq k \leq 7. \end{cases}
\end{displaymath}
Clearly 
\begin{equation} \label{eq_I_commutes_with_P}
 I^* \circ P_k=P_{7-k} \circ I^*.
\end{equation}

\begin{Proposition} \label{prop_eigenvectors_k_small}
 The eigenvalues and eigenvectors of 
 \begin{displaymath}
 P_k:\Curv_k(\h^2)^{\spsp} \to \Curv_k(\h^2)^{\spsp} 
 \end{displaymath}
 for $0\leq k \leq 3$ are given by the following table. The last column indicates to which isotypical component of $\Curv_k$ the eigenvector belongs. 
\begin{displaymath}
 \begin{array}{c | c | l | l}
 k &  \text{eigenvalue of } P_k & \text{eigenvector of } P_k & \text{module}\\ \hline 
 0 & 5040 & v_1^0 & \Gamma_{0,0,0,0}\\
 1 & -720  & v_1^1 & \Gamma_{0,0,0,0}\\
 1 & -384 & v_2^1 & \Gamma_{2,2,0,0}\\
2 &  60 & v_1^2 & \Gamma_{0,0,0,0}\\
2 & 0 & v_2^2 & \Gamma_{2,2,0,0}\\
2  & 102 & v_3^2  & \Gamma_{2,2,0,0}\\
2 & -90 & v_4^2  & \Gamma_{4,2,2,0}\\
 3 & -4 & v_1^3 & \Gamma_{0,0,0,0}\\
3 &  0 & v_2^3 & \Gamma_{2,2,0,0}\\
3 &  -18 & v_3^3 & \Gamma_{2,2,0,0}\\
3 &  0 & v_4^3  & \Gamma_{4,2,2,0}\\
3 & 46 & v_5^3 & \Gamma_{4,2,2,0}\\
3 &  -24 & v_6^3 & \Gamma_{2,2,2,2} \\
3 &  -80 & v_7^3 & \Gamma_{6,2,2,-2}
\end{array}
\end{displaymath}
\end{Proposition}

\proof
In terms of our basic invariant forms, we have 
\begin{align*}
 I^* d\alpha & = d\alpha\\
 I^* \beta_{\mathbf i} & =-\gamma_{\mathbf i}, I^*\beta_{\mathbf j}=-\gamma_{\mathbf j}, I^*\beta_{\mathbf k}=-\gamma_{\mathbf k}\\
 I^* \gamma_{\mathbf i} & =\beta_{\mathbf i}, I^*\gamma_{\mathbf j}=\beta_{\mathbf j}, I^*\gamma_{\mathbf k}=\beta_{\mathbf k}\\
 I^* \theta_{0,\mathbf i} & =\theta_{2,\mathbf i}, I^*\theta_{0,\mathbf j}=\theta_{2,\mathbf j}, I^*\theta_{0,\mathbf k}=\theta_{2,\mathbf k}\\
 I^* \theta_{1,\mathbf i} & =-\theta_{1,\mathbf i}, I^*\theta_{1j}=-\theta_{1,\mathbf j}, I^*\theta_{1,\mathbf k}=-\theta_{1,\mathbf k}\\
 I^* \theta_{2,\mathbf i} & =\theta_{0,\mathbf i}, I^*\theta_{2,\mathbf j}=\theta_{0,\mathbf j}, I^*\theta_{2,\mathbf k}=\theta_{0,\mathbf k}.
\end{align*}
From this, it can be checked that the displayed curvature measures are indeed eigenvectors to the displayed eigenvalue of $P_k$.

It remains to show that the displayed modules are correct. For $k=0$, this is trivial. 

For $k=1$, since the eigenvalues of $P_1$ are different, each eigenvector belongs to either $\Curv_1[\Gamma_{0,0,0,0}]$ or $\Curv_1[\Gamma_{2,2,0,0}]$. Since $\mathcal{L}v_1^1=7v_1^0 \in \Curv_0[\Gamma_{0,0,0,0}]$, we must have $v_1^1 \in \Curv_1[\Gamma_{0,0,0,0}]$ and $v_2^1 \in \Curv_1[\Gamma_{2,2,0,0}]$.  

For $k=2$, the argument is similar using Lemma \ref{lemma_derivation_operator}.

Let us study $k=3$ now. Since $\mathcal{L}v^3_1=-\frac52 v^2_1 \in \Curv_2[\Gamma_{0,0,0,0}], \mathcal{L}v^3_3=-3 v^2_3 \in \Curv_2[\Gamma_{2,2,0,0}], \mathcal{L}v^3_5=-3 v^2_4 \in \Curv_2[\Gamma_{4,2,2,0}]$ and the corresponding eigenvalues appear with multiplicity $1$, we deduce that these measures belong to the displayed isotypical components.  

For $v^3_2$ and $v^3_4$, the argument is more involved. Since both are eigenvectors to the same eigenvalue $0$, we cannot conclude that each of them belongs to an isotypical component. We only know that some linear combination belongs to one isotypical component and some other linear combination belongs to some other isotypical component:
\begin{displaymath}
\lambda_1 v_2^3+\lambda_2 v_4^3 \in \Curv_3[\Gamma_{2,2,0,0}], \quad  \tau_1 v_2^3 + \tau_2 v_4^3 \in \Curv_3[\Gamma_{4,2,2,0}]
\end{displaymath}
with $(\lambda_1,\lambda_2),(\tau_1,\tau_2) \neq (0,0)$. Applying $\mathcal{L}$ to the first equation yields $\lambda_1 \mathcal{L}v_2^3+\lambda_2 \mathcal{L}v_4^3 \in \Curv_2[\Gamma_{2,2,0,0}]$. Since $\mathcal{L}v^3_2=-42 v^2_3 \in \Curv_2[\Gamma_{2,2,0,0}]$ and $\mathcal{L}v^3_4=270 v^2_4 \in \Curv_2[\Gamma_{4,2,2,0}]$, this implies that $\lambda_2=0$. A similar argument shows that $\tau_1=0$. Hence $v^3_2 \in \Curv_3[\Gamma_{2,2,0,0}]$ and $v^3_4 \in \Curv_3[\Gamma_{4,2,2,0}]$ as displayed.

It remains to decide whether $v^3_6 \in \Curv_3[\Gamma_{2,2,2,2}], v^3_7 \in \Curv_3[\Gamma_{6,2,2,-2}]$ or $v^3_7 \in \Curv_3[\Gamma_{2,2,2,2}], v_3^6 \in \Curv_3[\Gamma_{6,2,2,-2}]$. The module $\Gamma_{2,2,2,2}$ enters the decomposition of $\Curv_3(\R^8)$ with multiplicity $1$, hence $P_3$ acts as a scalar on it. This scalar must be $-24$ in the first case and $-80$ in the second case. 

Identify $\R^8 \cong \C^4$. A direct computation, using \cite[Prop. 4.6]{bernig_fu_hig}, shows that the eigenvalues of $P_3$ restricted to $\Curv_3(\C^4)^{\mathrm{U}(4)}$ are $-4,0,-18,-24$, in particular $-80$ is not an eigenvalue. Moreover, $\Gamma_{2,2,2,2}$ appears in  $\Curv_3(\C^4)^{\mathrm{U}(4)}$. Therefore $P_3$ acts by the scalar $-24$ on $\Curv_3(\R^8)[\Gamma_{2,2,2,2}]$, which implies  that $v^3_6 \in \Curv_3[\Gamma_{2,2,2,2}]$ and $v^3_7 \in \Curv_3[\Gamma_{6,2,2,-2}]$.
\endproof

\begin{Corollary} \label{cor_eigenvectors_k_large}
 For $4 \leq k \leq 7$, we have the following eigenvalues, eigenvectors and   isotypical components
\begin{displaymath}
\begin{array}{c | c | l | l}
 k &  \text{eigenvalue of } P_k & \text{eigenvector of } P_k & \text{module}\\ \hline  
4 & -4 & v_1^4 & \Gamma_{0,0,0,0}\\
4 &  0 & v_2^4 & \Gamma_{2,2,0,0}\\
4 &  -18 & v_3^4 & \Gamma_{2,2,0,0}\\
4 & 0 & v_4^4 & \Gamma_{4,2,2,0}\\
4 &  46 & v_5^4 & \Gamma_{4,2,2,0}\\
4 &  -24 & v_6^4 & \Gamma_{2,2,2,2} \\
4 &  -80 & v_7^4 & \Gamma_{6,2,2,-2}\\
 5 & 60 & v_1^5 & \Gamma_{0,0,0,0}\\
5 & 0 & v_2^5 & \Gamma_{2,2,0,0}\\
5 & 102 & v_3^5  & \Gamma_{2,2,0,0}\\
5 & -90 & v_4^5 & \Gamma_{4,2,2,0}\\
6 & -720 & v_1^6 & \Gamma_{0,0,0,0}\\
6 & -384 & v_2^6 & \Gamma_{2,2,0,0}\\
7 & 5040 & v_1^7 & \Gamma_{0,0,0,0}
\end{array}
\end{displaymath}
\end{Corollary}

\proof
This follows directly from Proposition \ref{prop_eigenvectors_k_small} and \eqref{eq_I_commutes_with_P}: if $v$ is an eigenvector of $P_k$ belonging to  $\Curv_k[\Gamma_\lambda]$, then $I^*v$ is an eigenvector of $P_{7-k}$, with the same eigenvalue, belonging to  $\Curv_{7-k}[\Gamma_\lambda]$.
\endproof

\begin{Proposition} \label{prop_klain_fcts_large_k}
 The Klain functions of the curvature measures $v_i^k, 4 \leq k \leq 7$ are given by the following table
\begin{align*}
\begin{array}{c | c}
 \text{measure } m & \text{Klain function of } \glob m\\ \hline  
 v_1^4 & 6\pi^2 f_{4,0}\\
 v_2^4 &  0\\
 v_3^4 &  -\frac{3}{14}\pi^2(7f_{4,1}-18f_{4,0})\\
 v_4^4 &  0\\
 v_5^4 & \frac{23}{270}\pi^2 (20f_{4,3}+8f_{4,2}-43f_{4,1}+66f_{4,0})\\
 v_6^4 &  \frac25 \pi^2 (6f_{4,2}-f_{4,1})\\
 v_7^4 & -\frac{\pi^2}{54} (63f_{4,4}-161f_{4,3}-194f_{4,2}+226f_{4,1}-210f_{4,0})\\
 v_1^5 & -24 \pi f_{3,0}\\
 v_2^5  & 0\\
 v_3^5  & \frac{34}{63}\pi (7f_{3,1}-9f_{3,0})\\
 v_4^5 & -\frac{2}{9}\pi (16f_{3,2}-17f_{3,1}+15f_{3,0})\\
  v_1^6 & -12 \pi f_{2,0}\\
 v_2^6 &  -3\pi (7f_{2,1}-3f_{2,0})\\
  v_1^7 & -12 f_{1,0}.
\end{array}
\end{align*}
\end{Proposition}

\proof
By Theorem \ref{thm_alesker_bernig_schuster}, the decomposition of $\Val_k$ into irreducible $\mathrm{SO}(n)$-modules is multiplicity free. Moreover, if $\Gamma_\lambda$ is an irreducible $\mathrm{SO}(n)$-module entering this decomposition, then it contains at most one $\spsp$-invariant element, up to scaling. The Klain function of such a valuation was computed in \cite[Corollary 5.5]{bernig_solanes}, it corresponds (again up to scaling) to the right hand column. 

From these remarks it follows that the above table is correct up to scaling. To prove that the scaling is correct as well, we evaluate $\glob m$ and the function on the right hand side on a well-chosen $k$-plane $E^k$. This plane is defined as the direct product of a quaternionic line and an orthogonal real $(k-4)$-dimensional space. The stabilizer of the quaternionic line is $\Sp(1)\Sp(1)$, which acts on the orthogonal complement of the quaternionic line as $\mathrm{SO}(4)$. The subgroup which fixes also the $(k-4)$-dimensional space is therefore $\mathrm{SO}(8-k)$, and acts transitively on the $(7-k)$-dimensional unit sphere in the  orthogonal complement of $E^k$. If we plug in an orthogonal basis of $E^k$ into one of the forms $\Lambda_{k,i}$, we will therefore obtain an {\it $\mathrm{SO}(8-k)$-invariant} $(7-k)$-form on the sphere, i.e. just a multiple of the volume form. It is now straightforward to compute the Klain function of the globalization of $\Lambda_{k,i}$ at the given 
$k$-space $E^k$. 

We just illustrate this with an example in degree $5$. The valuations $\glob \Lambda_{5,1},\glob \Lambda_{5,2},\glob \Lambda_{5,4}$ evaluate at $E^5$ to $0$, while  $\Lambda_{5,3}$ evaluates to $-24\pi$. Therefore $\Kl \glob v_1^5(E^5)=-24 \pi$. 

On the other hand, on this special $5$-plane $E^5$, we have $\lambda_{12}=\lambda_{13}=\lambda_{23}=1$, and hence $f_{3,0}(E)=1$. Therefore $\glob v_1^5=-24 \pi f_{3,0}$.  
\endproof

For $0 \leq k \leq 3$, this method does not work since the stabilizer of a $k$-plane $E$ is no longer transitive on the sphere in the orthogonal complement.  

\begin{Proposition} \label{prop_klain_functions_k_small}
 The Klain functions of the curvature measures $v_i^k, 0 \leq k \leq 3$ are given by the following table
\begin{align*}
\begin{array}{c | c}
 \text{measure } m & \text{Klain function of } \glob m\\ \hline
 v_1^0 & -2\pi^4 f_{0,0}\\
 v_1^1 & -\frac{32}{5} \pi^3 f_{1,0}\\
 v_2^1 & 0\\ 
 v_1^2 & 6\pi^3 f_{2,0}\\
 v_2^2 & -\frac{1}{12}  \pi^3 (7f_{2,1}-3f_{2,0})\\
 v_3^2 &  -\frac{1}{63}\pi^3 (7f_{2,1}-3f_{2,0})\\
 v_4^2 & 0\\
 v_1^3 & -8\pi^2 f_{3,0}\\
 v_2^3 & \frac{8}{15}\pi^2 (7f_{3,1}-9f_{3,0})\\
 v_3^3 & \frac{4}{105}\pi^2 (7f_{3,1}-9f_{3,0})\\
 v_4^3 & - \frac{8}{21}\pi^2 (16f_{3,2}-17f_{3,1}+15f_{3,0})\\
 v_5^3 & \frac{4}{945}\pi^2 (16f_{3,2}-17f_{3,1}+15f_{3,0})\\
 v_6^3 & 0\\
 v_7^3 & 0
 \end{array}
 \end{align*}
\end{Proposition}

\proof
We illustrate the proof only for one case, namely $\glob v_1^3$, the other cases being similar. By Proposition \ref{prop_klain_fcts_large_k}  and its proof, we know that there is a constant $c$ such that  
\begin{displaymath}
 \Kl \circ \mathbb F \circ \glob v_1^3=c f_{3,0}=-\frac{c}{24 \pi} \Kl \circ \glob v_1^5.
\end{displaymath}

By the injectivity of the Klain map, $\mathbb F \circ \glob v_1^3=-\frac{c}{24 \pi} \glob v_1^5$.

We apply $\Lambda^2$ to this equation, using Corollary \ref{cor_multipliers} on the left hand side and Lemma \ref{lemma_derivation_operator} on the right hand side and obtain
\begin{displaymath}
 8\pi \glob(v_1^3)=\Lambda^2 \circ \mathbb F \circ \glob (v_1^3) =-\frac{c}{24 \pi} 24 \glob(v_1^3),
\end{displaymath}
which implies $c=-8\pi^2$ (since $\mathcal L(v^3_1)\neq 0$, we have  $\glob(v_1^3) \neq 0$).
\endproof

\section{The algebra structure}
\label{sec_algebra_structure}

Let us first describe the general strategy. We have obtained two different descriptions of invariant valuations: first in terms of Klain functions, second in terms of differential forms. The dictionary between these two languages is given in the previous section. 

The convolution product can be easily computed in terms of forms (see the formulas in \cite{bernig_fu06}). On the other hand, the Alesker-Fourier transform is easy to compute given the Klain functions. Since the Alesker-Fourier transform intertwines convolution and product, we have all the information we need to compute the product.

\begin{Definition}
 Define valuations $\kappa_2 \in \Val_2^{\spsp},\nu_3 \in \Val_3^{\spsp}$, and $\kappa_4, \nu_4 \in \Val_4^{\spsp}$ by 
 \begin{align*}
  \Kl_{\kappa_2} & = 7f_{2,1}-3f_{2,0}\\
  \Kl_{\nu_3} & =16f_{3,2}-17f_{3,1}+15f_{3,0}\\
  \Kl_{\kappa_4} & =6f_{4,2}-f_{4,1}\\
  \Kl_{\nu_4} & = 63f_{4,4}-161f_{4,3}-194f_{4,2}+226f_{4,1}-210f_{4,0},
 \end{align*}
and let $t:=\frac{2}{\pi} \mu_1 \in \Val_2^{\SO(8)}$. 
\end{Definition}

It is known that $t^i=\frac{i! \omega_i}{\pi^i} \mu_i$.

\begin{Proposition}\label{prop_kappanu_basis}
A basis of the vector space $\Val^{\spsp}$ is given by 
\begin{displaymath}
 \chi=t^0,t,t^2,\kappa_2,t^3,t\kappa_2,\nu_3,t^4,t^2\kappa_2,\kappa_4,t\nu_3,\nu_4,t^5,t^3\kappa_2,t^2\nu_3,t^6,t^4 \kappa_2,t^7,t^8.
\end{displaymath}
\end{Proposition}

\proof
It is enough to check that for each $k$ and each irreducible $\mathrm{SO}(8)$-module entering the decomposition of $\Val_k(\h^2)$ and containing an $\spsp$-invariant element, there exists an element from the given list belonging to it. 

Let us check this for $\Gamma_{4,2,2,0}$ which appears in the cases $k=3,4,5$. By Propositions \ref{prop_eigenvectors_k_small} and  \ref{prop_klain_functions_k_small}, $\nu_3=-\frac{21}{8 \pi^2} \glob(v_4^3)$ belongs to $\Val_3[\Gamma_{4,2,2,0}]$. This implies that $t \nu_3 \in \Val_4[\Gamma_{4,2,2,0}]$ and $t^2 \nu_3 \in \Val_5[\Gamma_{4,2,2,0}]$. Since $\Lambda^2 \circ L^2$ acts by a non-zero scalar on $\Val_3[\Gamma_{4,2,2,0}]$ (see Corollary \ref{cor_multipliers}), we have $t \nu_3 \neq 0, t^2 \nu_3 \neq 0$. 

The other cases are similar.  
\endproof

\begin{Proposition} \label{prop_globalization_map}
 The globalization map is given by the following table
 \begin{displaymath}
  \begin{array}{c | c | c || c |c | c}
  \text{measure} & \text{globalization} & \text{module} & \text{measure} & \text{globalization} & \text{module}\\ \hline
  v_1^0 &  -2\pi^4 \chi & \Gamma_{0,0,0,0}  & v_1^4 & \frac{\pi^4}{2}t^4 & \Gamma_{0,0,0,0}\\
  v_1^1 & -\frac{16}{5}\pi^4 t  & \Gamma_{0,0,0,0} & v_2^4 & 0 & \Gamma_{2,2,0,0}\\
  v_2^1 & 0 & \Gamma_{2,2,0,0} & v_3^4 & -\frac{3}{14}\pi^3 t^2 \kappa_2 & \Gamma_{2,2,0,0}\\
  v_1^2 & 3\pi^4 t^2 & \Gamma_{0,0,0,0} & v_4^4 & 0 & \Gamma_{4,2,2,0}\\
  v_2^2 & -\frac{1}{12}\pi^3 \kappa_2 & \Gamma_{2,2,0,0} & v_5^4 & \frac{46}{135}\pi^2 t \nu_3 & \Gamma_{4,2,2,0}\\
  v_3^2 & -\frac{1}{63}\pi^3 \kappa_2 & \Gamma_{2,2,0,0} & v_6^4 & \frac{2}{5} \pi^2 \kappa_4 & \Gamma_{2,2,2,2}\\
  v_4^2 & 0 & \Gamma_{4,2,2,0} & v_7^4 & -\frac{\pi^2}{54} \nu_4 & \Gamma_{6,2,2,-2}\\ 
  v_1^3 & -\pi^4 t^3 & \Gamma_{0,0,0,0} & v_1^5 & -\frac38 \pi^4 t^5 & \Gamma_{0,0,0,0}\\
  v_2^3 & \frac25 \pi^3 t \kappa_2 & \Gamma_{2,2,0,0} & v_2^5 & 0 & \Gamma_{2,2,0,0}\\ 
  v_3^3 & \frac{1}{35} \pi^3 t \kappa_2 & \Gamma_{2,2,0,0} & v_3^5 & \frac{85}{504} \pi^3 t^3 \kappa_2 & \Gamma_{2,2,0,0}\\
  v_4^3 & -\frac{8}{21} \pi^2 \nu_3 & \Gamma_{4,2,2,0} & v_4^5 & -\frac{7}{18} \pi^2  t^2 \nu_3 & \Gamma_{4,2,2,0}\\
  v_5^3 & \frac{4}{945} \pi^2 \nu_3 & \Gamma_{4,2,2,0} & v_1^6 & -\frac{\pi^4}{10} t^6 & \Gamma_{0,0,0,0}\\
  v_6^3 & 0 & \Gamma_{2,2,2,2} & v_2^6 & -\frac34 \pi^3 t^4 \kappa_2 & \Gamma_{2,2,0,0}\\
  v_7^3 & 0 & \Gamma_{6,2,2,-2} & v_1^7 & -\frac{\pi^4}{64} t^7 & \Gamma_{0,0,0,0}
  \end{array}
\end{displaymath}
\end{Proposition}

\proof
Again, we only illustrate the proof in one case, with the other cases being similar. 

We must have $\glob v_2^3=c t \kappa_2$ for some constant $c$. We apply the derivation operator $\Lambda$ to both sides. From Lemma \ref{lemma_derivation_operator} it follows that $\Lambda \circ \glob v_2^3=-8 \glob v_2^2$, while Corollary \ref{cor_multipliers} implies that $\Lambda \circ t \kappa_2=\frac53 \kappa_2$. From this it follows that $\glob v_2^3=\frac25 \pi^3 t \kappa_2$. 

The globalization of $v_3^3$ follows from this computation and the relation $\glob v_2^{3}=14 \glob v_3^{3}$.
\endproof

\begin{Corollary} \label{cor_fourier_transforms_kappa}
 The Alesker-Fourier transform is given by the following table.
 \begin{displaymath}
  \begin{array}{c | c}
   \text{valuation} & \text{Fourier transform} \\ \hline
   t & \frac{\pi^3}{384} t^7\\
   t^2 & \frac{\pi^2}{60}t^6\\
   \kappa_2 & \frac{\pi^2}{4}t^4 \kappa_2\\
   t^3 & \frac{\pi}{8}t^5\\
   t\kappa_2 & \frac{5}{12}\pi t^3 \kappa_2\\
   \nu_3 & \frac74 \pi t^2 \nu_3\\
\end{array}
\end{displaymath}
Together with the fact that $\mathbb F$ acts trivially on $\Val_4^{\spsp} $(see Proposition \ref{prop_multiplier_fourier}) and that $\mathbb{F} \circ \mathbb{F}=\mathrm{id}$ on $\Val^{\spsp}$, this determines the action of $\mathbb F$ completely. 
\end{Corollary}

\proof
Using Proposition \ref{prop_globalization_map}, we may write a valuation as globalization of some curvature measure, and Propositions \ref{prop_klain_fcts_large_k} and \ref{prop_eigenvectors_k_small} give the Klain function of the valuation. Then the defining property \eqref{eq_fourier} of the Alesker-Fourier transform gives the values stated in the table.  
\endproof

\begin{Theorem} \label{thm_algebra_structureI}
 The algebra structure is given by 
\begin{displaymath}
 \begin{array}{c | c | c | c | c}
  & \kappa_2 & \kappa_4 & \nu_3 & \nu_4\\ \hline
  \kappa_2 & \frac{8}{5}\pi^2 t^4-\frac{59}{8}\pi t^2 \kappa_2-\frac{49}{30} t \nu_3+\frac{98}{5}\kappa_4 & \frac{49}{4} \pi^2 t^4 \kappa_2 & \frac{63}{32}\pi^2 t^3 \kappa_2-\frac{743}{24}\pi t^2 \nu_3 & 0\\  \hline
  \kappa_4 & & \pi^4 t^8 & 0 & 0\\  \hline
  \nu_3 & & & -\frac{27}{14}\pi^4 t^6+\frac{33435}{896}\pi^3 t^4 \kappa_2 & 0\\  \hline
  \nu_4 & & & & 864 \pi^4 t^8
 \end{array}
\end{displaymath}
\end{Theorem}

\proof
Once again, we content ourselves with proving one of these equations. Let us compute $\kappa_2 \cdot \kappa_2$.  

By Corollary \ref{cor_fourier_transforms_kappa} and Proposition \ref{prop_globalization_map}, we have $\mathbb F \kappa_2=\frac{\pi^2}{4} t^4 \kappa_2=-\frac{1}{3\pi} \glob v_2^6$. Since $v_2^6$ is presented by the differential form $-\frac32 \Lambda_{6,1}+\Lambda_{6,2}$,  equation \eqref{eq_convolution} implies that \[
 \kappa_2^2=\frac{1}{9\pi^2}\mathbb F\circ\glob \circ  *_1^{-1}\left(*_1\left(-\frac32 \Lambda_{6,1}+\Lambda_{6,2}\right) \wedge *_1D\left(-\frac32 \Lambda_{6,1}+\Lambda_{6,2}\right)\right).
\]

Let us recall the definition of the operator $*_1$ acting on differential forms on the sphere bundle of an $n$-dimensional oriented euclidean vector space $V$. 

Let $*_V$ be the Hodge star acting on forms on $V$. Then $*_1$ is defined as the $C^\infty(SV)$-linear operator such that 
\begin{displaymath}
*_1(\eta_1 \wedge \eta_2)=(-1)^{\binom{n-\deg \eta_1}{2}} (*_V\eta_1) \wedge \eta_2, \quad \eta_1 \in \Omega^*(V), \eta_2 \in \Omega^*(S^{n-1}).
\end{displaymath}

A computation in coordinates gives
\begin{align*}
 *_1 \Lambda_{6,1} & = 3\alpha \wedge \theta_s\\
 *_1 \Lambda_{6,2} & = 6 \alpha \wedge \phi_{1,1}.
\end{align*}
Further computation yields 
\begin{align*}
 *_1 \Lambda_{5,1} & = \frac12 \alpha \wedge \theta_s^2\\ 
 *_1 \Lambda_{5,2} & = 3 \alpha \wedge \phi_{1,1} \wedge \theta_s\\ 
 *_1 \Lambda_{5,3} & = 3 \alpha \wedge \phi_{1,1}^2\\ 
 *_1 \Lambda_{5,4} & =-\frac16 \alpha \wedge \Psi
\end{align*}
with
\begin{displaymath}
 \Psi= \left\{2\beta_{\mathbf i} \wedge \beta_{\mathbf j} \wedge \theta_{0,\mathbf k}-\beta_{\mathbf i} \wedge \gamma_{\mathbf j} \wedge \theta_{1,\mathbf k}-\gamma_{\mathbf i} \wedge \beta_{\mathbf j} \wedge \theta_{1,\mathbf k}+2\gamma_{\mathbf i} \wedge \gamma_{\mathbf j} \wedge \theta_{2,\mathbf k}\right\}.
\end{displaymath}
Hence, using $*_1(\alpha \wedge \tau)=-i_T*_1\tau$, we obtain
\begin{align*}
& *_1 D\left(-\frac32 \Lambda_{6,1}+\Lambda_{6,2}\right)  = *_1\alpha \wedge (54 \Lambda_{5,1}-10\Lambda_{5,2}+2 \Lambda_{5,3}-126\Lambda_{5,4})\\
 & = -27 \theta_s^2+30 \phi_{1,1} \wedge \theta_s -6 \phi_{1,1}^2-21\Psi. 
\end{align*}
Finally,
\begin{align*}
 *_1 \Lambda_{4,1} & = \frac12 \alpha \wedge \theta_s^3\\ 
 *_1 \Lambda_{4,2} & = \frac12 \alpha \wedge \phi_{1,1} \wedge \theta_s^2\\ 
 *_1 \Lambda_{4,4} & = \frac32 \alpha \wedge \phi_{1,1}^2 \wedge \theta_s\\ 
 *_1 \Lambda_{4,5} & =  \alpha \wedge \phi_{1,1}^3\\ 
 *_1 \Lambda_{4,6} & = -\frac16\alpha \wedge \theta_s \wedge\Psi\\
 *_1 \Lambda_{4,7} & = -\frac16\alpha \wedge \phi_{1,1}\wedge\Psi.
\end{align*}

It follows that $\kappa_2^2$ is presented by the differential form 
\begin{align*}
 &\frac{1}{9\pi^2} *_1^{-1}\left[\alpha \wedge \left(-\frac92 \theta_s+6\phi_{1,1}\right) \wedge \left(-27 \theta_s^2+30 \phi_{1,1} \wedge \theta_s-6\phi_{1,1}^2-21\Psi\right)\right]\\
 &=\frac{27}{\pi^2}\Lambda_{4, 1}-\frac{66}{\pi^2}\Lambda_{4, 2}+\frac{46}{3\pi^2}\Lambda_{4, 4}-\frac{4}{\pi^2}\Lambda_{4, 5}-\frac{63}{\pi^2}\Lambda_{4, 6}+\frac{84}{\pi^2}\Lambda_{4, 7}
\end{align*}
which globalizes to $\frac{8}{5}\pi^2 t^4-\frac{59}{8}\pi t^2 \kappa_2-\frac{49}{30} t \nu_3+\frac{98}{5}\kappa_4$.
\endproof

As an immediate consequence, we get the following values for the Alesker-Poincar\'e pairing $\mathrm{pd}$
\begin{align}\label{pd_values}
\mathrm{pd}(t^i,t^{8-i}) & =\frac{1680}{\pi^4},\qquad i=0,\ldots,8,\notag\\ 
\mathrm{pd}(t^i\kappa_2,t^{4-i}\kappa_2) & =\frac{2688}{\pi^2},\qquad i= 0,\ldots,4, \\ 
\mathrm{pd}(t^i\nu_3,t^{2-i}\nu_3) & =-3240,\qquad  i=0,1,2\notag \\
\mathrm{pd}(\kappa_4,\kappa_4) & =1680 \notag\\
\mathrm{pd}(\nu_4,\nu_4) & =1451520.\notag
\end{align}
Since $\mathrm{pd}(\phi,\psi)=0$ if $\phi,\psi$  belong to  non-isomorphic irreducible components of $\Val$, all other pairings of elements from the basis of Proposition \ref{prop_kappanu_basis} are zero.

\begin{Proposition} \label{prop_kinform_kappa}
\begin{align*}
k(\chi) & =\frac{\pi^4}{1680}\left(2t^8\odot\chi +2t^7\odot t+2t^6\odot t^2+2t^5\odot t^3+t^4\odot t^4\right)\\
& \quad +\frac{\pi^2}{2688}\left(2(t^4\kappa_2)\odot\kappa_2+2(t^3\kappa_2)\odot(t\kappa_2)+(t^2\kappa_2)\odot(t^2\kappa_2)\right)\\
 & \quad -\frac{1}{3240}(2(t^2\nu_3)\odot\nu_3+(t\nu_3)\odot(t\nu_3))+\frac{1}{1680}\kappa_4\odot\kappa_4+\frac{1}{1451520}\nu_4\odot\nu_4.
\end{align*}
\end{Proposition}

\begin{proof}
Since $k(\chi)\in\Val^{\spsp} \otimes \Val^{\spsp} \cong\mathrm{Hom}(\Val^{\spsp *},\Val^{\spsp})$ is the inverse of $\mathrm{pd}$, the stated formula follows from \eqref{pd_values}.
\end{proof}

\proof[Proof of Theorem \ref{mainthm_kinform}]
Using Propositions \ref{prop_globalization_map}, \ref{prop_klain_functions_k_small} and \ref{prop_klain_fcts_large_k}, we find the following table
\begin{displaymath}
  \begin{array}{c | c}
   \text{valuation in $t-\kappa$ basis} & \text{valuation in $\varphi$ basis}  \\ \hline
   t^i, i=0,\ldots,8 & \frac{i! \omega_i}{\pi^i} \varphi_i\\
   \kappa_2 & 7\varphi_{2,1}-3\varphi_{2,0}\\
   t \kappa_2 & \frac{28}{3 \pi} \varphi_{3,1}-\frac{12}{\pi} \varphi_{3,0}\\
   t^2 \kappa_2 & \frac{7}{\pi} \varphi_{4,1}-\frac{18}{\pi} \varphi_{4,0}\\
   t^3 \kappa_2 & \frac{112}{5 \pi^2}\varphi_{5, 1}-\frac{144}{5\pi^2} \varphi_{5, 0}\\
   t^4 \kappa_2 & \frac{28}{\pi^2} \varphi_{6, 1}-\frac{12}{\pi^2} \varphi_{6, 0} \\
   \nu_3 & 16\varphi_{3,2}-17\varphi_{3,1}+15\varphi_{3,0}\\
   t \nu_3 & 5 \varphi_{4, 3}+2\varphi_{4, 2}-\frac{43}{4} \varphi_{4, 1}+\frac{33}{2} \varphi_{4, 0}\\
   t^2 \nu_3 & \frac{64}{7\pi} \varphi_{5, 2}-\frac{68}{7\pi} \varphi_{5, 1}+\frac{60}{7\pi} \varphi_{5, 0} \\
   \kappa_4 & 6\varphi_{4,2}-\varphi_{4,1}\\
   \nu_4 & 63\varphi_{4,4}-161\varphi_{4,3}-194\varphi_{4,2}+226\varphi_{4,1}-210\varphi_{4,0}
\end{array}
\end{displaymath}
The statement of the theorem follows from the previous proposition by replacing these values. 
\endproof

\begin{Proposition} \label{prop_ts_basis}
 Endow the quaternionic projective line $\mathbb{HP}^1$ with the invariant probability measure and define 
\begin{align*}
 t(K) & := \frac{35}{12\pi} \int_{\mathbb{H}P^1} \mu_1(\pi_EK)dE\\
 s(K) & := \int_{\mathbb{H}P^1} \mu_2(\pi_EK)dE\\
 v(K) & := \int_{\mathbb{H}P^1} \mu_3(\pi_EK)dE\\
 u(K) & := \int_{\mathbb{H}P^1} \mu_4(\pi_EK)dE. 
\end{align*}
Here $\mu_i$ is the $i$-th intrinsic volume. Then 
\begin{align*}
 s & =  \frac{3\pi}{14} t^2+\frac{1}{56}\kappa_2\\
 v & =\frac{\pi^2}{35} t^3+\frac{3\pi}{280}t\kappa_2-\frac{1}{630}\nu_3\\
  u & = \frac{\pi^2}{140}t^4+\frac{\pi}{112}t^2\kappa_2-\frac{1}{180}t\nu_3+\frac{1}{140}\kappa_4+\frac{1}{10080}\nu_4.
\end{align*}
\end{Proposition}

\proof
In order to find the coefficients in 
\begin{displaymath}
 u=a_0t^4+a_1 t^2\kappa_2+a_2 t\nu_3+a_3\kappa_4+a_4\nu_4,
\end{displaymath}
we note that,  
\begin{displaymath}
 \mathrm{pd}(u,\phi)=\Kl_{\phi}(\mathbb H\oplus 0),\qquad\forall \phi \in \Val_4^{\spsp}.
\end{displaymath}

Notice also that $\mathbb H\oplus 0\in \Gr_4$ corresponds to $\lambda_{ij}=1$ for all $i,j$, i.e. $f_{4,0}=1, f_{4,1}=6, f_{4,2}=3, f_{4,3}=12, f_{4,4}=24$. 

Taking $\phi=t^2\kappa_2$, we get
\[
 a_1\frac{2688}{\pi^2}=\mathrm{pd}(u,t^2\kappa_2)=\Kl_{t^2\kappa_2}(\mathbb H \oplus 0)=\frac{24}{\pi},
\]
and $a_1=\frac{\pi}{112}$. The other coefficients are obtained in the same way. 

Since $s=\frac{1}{2\pi} \Lambda^2 u, v =\frac12 \Lambda u$, the other equations follow from Proposition \ref{prop_globalization_map}, and Lemma \ref{lemma_derivation_operator}. 
\endproof

\begin{Theorem} \label{thm_algebra_structureII}
The valuations $t,s,v,u$ generate $\Val^{\spsp}$. Moreover, we have an algebra isomorphism
 \begin{displaymath}
  \Val^{\spsp} \cong \R[t,s,v,u]/I,
 \end{displaymath}
 where $I$ is the ideal generated by 
 \begin{align*}
&tk_4,tn_4,k_2n_3-\frac{63}{32} \pi^2 t^3 k_2+\frac{743}{24}\pi t^2 n_3,\\
&t^3n_3,k_2 k_4-\frac{49}{4}\pi^2 t^4 k_2,k_2n_4,n_3^2+\frac{27}{14}\pi^4 t^6-\frac{33435}{896} \pi^3 t^4 k_2,\\
&t^5k_2,k_4n_3,n_3n_4,\\
&k_4^2-\pi^4t^8,k_4n_4, n_4^2-864 \pi^4 t^8,
 \end{align*}
with 
 \begin{align*}
  k_2 & := -12 \pi t^2+ 56 s\\
  k_4 & :=   -\frac52 \pi^2 t^4 -  16  \pi t^2 s + 160   s^2-\frac{105}{2}  tv\\
  n_3 & := -63 \pi^2 t^3+ 378  \pi ts-  630  v\\
  n_4 & := -2340 \pi^2 t^4+ 17280  \pi t^2 s-  11520   s^2-  31500 tv+10080u.
 \end{align*}
\end{Theorem}

\proof
By Proposition \ref{prop_ts_basis}, the valuations $\kappa_2,\kappa_4,\nu_3,\nu_4$ can be expressed as polynomials in $t,s,v,u$. More precisely, 
\begin{align*}
\kappa_2 & =-12 \pi t^2+ 56 s\\
\kappa_4 & =  -\frac52 \pi^2 t^4 -  16  \pi t^2 s + 160  s^2-\frac{105}{2}  tv\\
\nu_3 & = -63 \pi^2 t^3+ 378  \pi ts-  630  v\\
\nu_4 & =-2340 \pi^2 t^4+ 17280  \pi t^2 s-  11520  s^2-  31500 tv+10080u.
\end{align*}

This proves that $t,s,v,u$ generate the algebra. 

The displayed polynomial equations are just rewritings of Theorem \ref{thm_algebra_structureI}.
\endproof
\def\cprime{$'$}


\end{document}